\documentclass[11pt]{article}

\usepackage{amssymb}
\usepackage{amsmath}
\usepackage{graphicx}
\usepackage{caption}
\usepackage{subcaption}
\usepackage{float}
\usepackage{dcolumn}
\usepackage{amsthm}
\usepackage{empheq}
\usepackage{multirow}
\usepackage{enumitem}
\usepackage{amsfonts}
\usepackage{color}
\usepackage[sort]{cite}
\usepackage{color}
\usepackage{xcolor}
\usepackage{colortbl}
\usepackage{adjustbox}
\usepackage{booktabs}

\usepackage{geometry}
\usepackage{hyperref}
\hypersetup{
	colorlinks=true,
	linkcolor=blue,
        citecolor=blue,
	filecolor=magenta,      
	urlcolor=cyan,
}

\usepackage{pgf}
\usepackage{tikz}
\usetikzlibrary{arrows.meta,decorations.pathmorphing,backgrounds,positioning,fit}

\numberwithin{equation}{section}

\def\Line(#1,#2)(#3,#4){\qbezier(#1,#2)(#1,#2)(#3,#4)}
\def\DashLine(#1,#2)(#3,#4){\qbezier[50](#1,#2)(#1,#2)(#3,#4)}

\newcommand{\inp}[2][]{\left(#1, #2\right)}
\newcommand{\gnp}[2][]{\langle#1, #2\rangle}

\newtheorem{remark}{Remark}[section]
\newtheorem{lemma}{Lemma}[section]

\newtheorem{theorem}{Theorem}[section]

\def\Tc{\mathcal{T}}

\def\Qc{\mathcal{Q}}

\def\ERT{\mathcal{ERT}}
\def\RT{\mathcal{R\!T}}

\def\Pc{\mathcal{P}}

\def\X{\mathbb{X}}
\def\W{\mathbb{W}}
\def\R{\mathbb{R}}
\def\M{\mathbb{M}}
\def\S{\mathbb{S}}
\def\N{\mathbb{N}}

\def\Q{\mathbb{Q}}

\def\x{\mathbf{x}}

\def\s{\sigma}
\def\t{\tau}
\def\g{\gamma}

\def\r{\mathbf{r}}

\def\As{A_{\s\s}}
\def\Au{A_{\s u}}
\def\Ag{A_{\s\g}}

\def\Eh{\hat{E}}
\def\Ah{\hat{A}}

\def\rh{\hat{\r}}

\def\nh{\hat{n}}

\def\xh{\hat{x}}
\def\yh{\hat{y}}
\def\zh{\hat{z}}

\def\th{\hat\tau}

\def\ch{\hat{\chi}}

\def\Xh{\hat{\X}}
\def\Vh{\hat{V}}

\def\skew{\operatorname{Skew}}

\def\curl{\operatorname{curl}}
\def\dvr{\operatorname{div}}          
\def\dvrg{\operatorname{div}}  

\def\DF{D\!F}

\def\Om{\Omega}
\def\Gn{\Gamma_N}
\def\Gd{\Gamma_D}

\allowdisplaybreaks

\begin{document}

\title{Multipoint stress mixed finite element methods for elasticity on cuboid grids}

\author{Ibrahim Yazici\thanks{Department of Mathematics, University of
    Pittsburgh, Pittsburgh, PA 15260, USA;~{\tt \{iby2@pitt.edu, yotov@math.pitt.edu\}}. Partially supported by NSF grants DMS-2111129 and DMS--2410686.}
\and Ivan Yotov\footnotemark[1]~}

\date{}
\maketitle
\begin{abstract}
  We develop multipoint stress mixed finite element methods for linear elasticity with weak stress symmetry on cuboid grids, which can be reduced to a symmetric and positive definite cell-centered system. The methods employ the lowest-order enhanced Raviart-Thomas finite element space for the stress and piecewise constant displacement. The vertex quadrature rule is employed to localize the interaction of stress degrees of freedom, enabling local stress elimination around each vertex. We introduce two methods. The first method uses a piecewise constant rotation, resulting in a cell-centered system for the displacement and rotation. The second method employs a continuous piecewise trilinear rotation and the vertex quadrature rule for the asymmetry bilinear forms, allowing for further elimination of the rotation and resulting in a cell-centered system for the displacement only. Stability and error analysis is performed for both methods. For the stability analysis of the second method, a new auxiliary $H(\curl)$-conforming matrix-valued space is constructed, which forms an exact sequence with the stress space. A matrix-matrix inf-sup condition is shown for the curl of this auxiliary space and the trilinear rotation space. First-order convergence is established for all variables in their natural norms, as well as second-order superconvergence of the displacement at the cell centers. Numerical results are presented to verify the theory.
\end{abstract}

\section{Introduction}

Mixed finite element (MFE) methods for stress-displacement elasticity formulations offer accurate stress with local momentum conservation and locking-free approximations for nearly incompressible materials \cite{Boffi-Brezzi-Fortin}. Numerous studies have explored both strong \cite{arnold2005rectangular, arnold2002mixed} and weak \cite{arnold2015mixed, arnold2007mixed, lee2016towards, awanou2013rectangular,CGG,GopGuz-second,JunLee-mesh-dep,JunLee-optimal} stress symmetry within these methods. However, a drawback is that they yield algebraic systems of the saddle point type, which can be computationally costly to solve. To alleviate this issue, multipoint stress mixed finite element (MSMFE) methods that can be reduced to symmetric and positive definite cell-centered systems have been developed for simplicial \cite{msmfe-simpl} and quadrilateral \cite{msmfe-quads} grids. These methods are related to the multipoint stress approximation (MPSA) method \cite{Nordbotten-MPSA,keilegavlen2017finite} and inspired by the multipoint flux mixed finite element (MFMFE) method for Darcy flow \cite{WheXueYot-nonsym, IngWheYot-sym, WheYot-mfmfe-2006}, which is closely related to the multipoint flux approximation (MPFA) method \cite{Aavatsmark-etal-mpfa,EdwardsRogers,KlausenWinther,AEKWY}. The MPFA method is a finite volume method obtained by eliminating fluxes around mesh vertices in terms of neighboring pressures. The MFMFE method employs the lowest order Brezzi-Douglas-Marini $\mathcal{BDM}_1$ spaces \cite{BDM} on simplicial and quadrilateral grids and an enhanced Brezzi-Douglas-Duran-Fortin $\mathcal{BDDF}_1$ space \cite{Boffi-Brezzi-Fortin} on hexahedral grids \cite{IngWheYot-sym}. A vertex quadrature rule is used, allowing for local velocity elimination and leading to a positive definite cell-centered system for pressures.

In this paper we develop two MSMFE methods for elasticity on cuboid grids that can be reduced to symmetric and positive definite cell-centered systems. We consider a formulation where the symmetry of the stress is imposed weakly using a Lagrange multiplier, which has a physical interpretation as rotation. We note that there have been relatively few works on mixed elasticity with weak stress symmetry on cuboid grids \cite{awanou2013rectangular,lee2016towards,JunLee-mesh-dep,JunLee-optimal}. The spaces we propose are different from those previously used. They are specifically defined to allow for local elimination of stress and, in the second method, rotation. Our first method, referred to as MSMFE-0, uses a stress-displacement-rotation triple $\X_h\times V_h \times \W_h$
based on the spaces $\mathcal{ERT}_0 \times Q_0 \times Q_0$, utilizing the lowest order enhanced Raviart-Thomas space $\mathcal{ERT}_0$ for stress and piecewise constant displacement and rotation. The enhanced Raviart-Thomas spaces of order $k$ are introduced in \cite{high-order-mfmfe} for higher order MFMFE methods. The construction involves enhancing the Raviart-Thomas spaces with curls of specifically chosen polynomials. In the lowest order case, the degrees of freedom can be chosen as the values of the normal components at any four points of each of the six faces. We choose these points to be the vertices, motivated by the vertex quadrature rule we employ. This localizes the interaction of the stress degrees of freedom around the vertices, resulting in a block-diagonal stress matrix. The stress is then locally eliminated, reducing the method to a symmetric and positive definite cell-centered system for displacement and rotation, which is smaller and easier to solve than the original system. Our second method, MSMFE-1, is based on the spaces $\mathcal{ERT}_0 \times Q_0 \times Q_1$ with continuous trilinear rotation. In this method, we employ the vertex quadrature rule for both stress-stress and stress-rotation bilinear forms. This allows for further local elimination of the rotation after the initial stress elimination, resulting in a symmetric and positive definite cell-centered system for displacement only.

We perform stability and error analysis for both MSMFE methods. The stability of the MSMFE-0 method follows easily from the argument in \cite{awanou2013rectangular} for the method based on $\mathcal{BDM}_1\times Q_0 \times Q_0$ on cuboid grids. However, the stability analysis for the MSMFE-1 method requires a different approach. We utilize the framework from \cite{arnold2015mixed,lee2016towards}, which is based on two conditions. One is that $\X_h \times V_h$ is a stable Darcy pair, which holds for the $\mathcal{ERT}_0 \times Q_0$ spaces \cite{high-order-mfmfe}. The second condition requires a construction of an auxiliary matrix-valued space $\Q_h$ such that $\curl \Q_h \subset \X_h$ and $\Q_h\times \W_h$ is a stable Stokes pair with the vertex quadrature rule. This is the key component of our analysis. In our case it is not possible to construct $H^1$-conforming space that satisfies a standard Stokes inf-sup condition, which was the approach in previously developed MSMFE methods on simplices \cite{msmfe-simpl} and quadrilaterals \cite{msmfe-quads}. Instead, we construct $H(\curl)$-conforming space $\Theta_h$, utilizing the covariant transformation, and set $\Q_h = (\Theta_h)^3$. The space $\Theta_h$ is of interest by itself, since it forms an exact sequence with $\mathcal{ERT}_0$, i.e., $\curl \Theta_h$ is the divergence-free subspace of $\mathcal{ERT}_0$. We establish that the matrix-valued pair $\Q_h\times \W_h$ satisfies a curl-based inf-sup condition with the vertex quadrature rule. Following the stability analysis, we establish for both methods first-order convergence for the stress in the $H(\dvrg)$-norm and for the displacement and rotation in the $L^2$-norm, as well as second-order superconvergence of the displacement at the cell centers. The theory is illustrated by numerical experiments, including a test showing locking-free behavior for nearly incompressible materials. Additionally, a modified version of the MSMFE-1 method based on scaled rotation is introduced, which is better suited for problems with discontinuous compliance tensors, and its performance is verified numerically.

The rest of the paper is organized as follows. The model problem and its MFE approximation are presented in Section 2. Sections 3 and 4 develop the two MSMFE methods and their stability, respectively. The error analysis is performed in Section 5. Numerical results are presented in Section 6.

\section{The model problem and its MFE approximation}

In this section, we recall the weak stress symmetry formulation of the elasticity system. Following this, we introduce its MFE approximation and a quadrature rule, which together serve as the foundation for the MSMFE methods discussed in the subsequent sections.

Let $\Om$ be a simply connected bounded domain of $\R^3$ occupied by a
linearly elastic body. We write $\M$, $\S$ and $\N$ for the spaces of real $3\times 3$ matrices, symmetric matrices and skew-symmetric matrices, respectively. We will utilize the usual divergence
operator $\dvr$ for vector fields. When applied to a matrix field, it produces a vector field by taking the divergence of each row. We will also use the usual curl operator applied to vector fields in three dimensions. For a matrix field in three dimensions, the curl operator produces a matrix field, by acting row-wise.

Throughout the paper, $C$ denotes a generic positive constant that is
independent of the discretization parameter $h$. We will also use the
following standard notation. For a domain $G \subset \R^3$, the
$L^2(G)$ inner product and norm for scalar, vector, or tensor valued functions
are denoted $\inp[\cdot]{\cdot}_G$ and $\|\cdot\|_G$,
respectively. The norms and seminorms of the Sobolev spaces
$W^{k,p}(G),\, k \in \R,\ p>0$ are denoted by $\| \cdot \|_{k,p,G}$ and
$| \cdot |_{k,p,G}$, respectively. The norms and seminorms of the
Hilbert spaces $H^k(G)$ are denoted by $\|\cdot\|_{k,G}$ and $| \cdot
|_{k,G}$, respectively. We omit $G$ in the subscript if $G = \Om$. For
a section of the domain or element boundary $S$ we
write $\gnp[\cdot]{\cdot}_S$ and $\|\cdot\|_S$ for the $L^2(S)$ inner
product (or duality pairing) and norm, respectively. We define
the spaces 
\begin{align*}
&H(\dvrg;\Om) = \{v \in L^2(\Om, \R^3) : \dvr v \in L^2(\Om)\},\\
&H(\dvrg;\Om,\mathbb{M}) = \{\tau \in L^2(\Om, \mathbb{M}) : \dvr \tau \in L^2(\Om,\mathbb{R}^3)\},\\
&H(\curl;\Om) = \{v \in L^2(\Om, \R^3) : \curl v \in L^2(\Om,\R^3)\},\\
&H(\curl;\Om,\mathbb{M}) = \{\tau \in L^2(\Om, \mathbb{M}) : \curl \tau \in L^2(\Om,\mathbb{M})\},
\end{align*}
equipped with the norms
\begin{align*}
&\|\tau\|_{\dvr} = \left( \|\tau\|^2 + \|\dvr \tau\|^2 \right)^{1/2}, \quad \|\tau\|_{\curl} = \left( \|\tau\|^2 + \|\curl \tau\|^2 \right)^{1/2}.
\end{align*}

Let $A = A(\x)$ be the compliance tensor describing the material
properties of the elastic body at each point $\x \in \Om $, which is a symmetric, bounded and uniformly
positive definite linear operator acting from $\S$ to $\S$. We also
assume that an extension of $A$ to an operator $\M \to \M$ still possesses the above properties.
Given a vector field $f \in L^2(\Omega,\R^3)$ representing body forces,
the equations of static linear elasticity in Hellinger-Reissner form determine
the stress $\sigma$ and the displacement $u$ satisfying the
constitutive and equilibrium equations, respectively,
\begin{align}
    A\s = \epsilon(u), \quad \dvr \s = f \quad \text{in } \Om, \label{elast-1}
\end{align}
together with the boundary conditions 
\begin{align}
    u = 0 \ \text{ on } \Gd, \quad  \s\,n = 0 \ \text{ on } \Gn, \label{elast-2}
\end{align}
where $\epsilon(u) = \frac{1}{2}\left(\nabla u + (\nabla u)^T\right)$ and $\partial\Om = \Gd \cup \Gn$.
To avoid issues with non-uniqueness, we assume that $\Gd \neq \emptyset$.

We consider a weak formulation for \eqref{elast-1}--\eqref{elast-2}, in
which the stress symmetry is imposed weakly, using the
Lagrange multiplier $\g = \skew(\nabla u)$,
$\skew(\tau) = \frac12(\tau - \tau^T)$, from the space of
skew-symmetric matrices:
find $(\s, u, \g) \in \X \times V \times \W$ such that:
\begin{align}
    \inp[A\s]{\tau} + \inp[u]{\dvr \tau} + \inp[\g]{\tau} &= 0, &\forall \tau &\in \X, \label{weak-1}\\
    \inp[\dvr \s]{v} &= \inp[f]{v}, &\forall v &\in V, \label{weak-2}\\
	\inp[\s]{w} &= 0, &\forall w &\in \W, \label{weak-3}
\end{align}
where the corresponding spaces are
$$\X = \left\{ \tau\in H(\dvrg;\Omega,\M) : \tau\,n = 0 \text{ on } \Gn  \right\}, \quad V = L^2(\Omega, \R^3), \quad \W = L^2(\Omega, \N).
$$
Problem \eqref{weak-1}--\eqref{weak-3} has a unique solution \cite{arnold2007mixed}. 

Define the invertible operators $S:\mathbb{M}\rightarrow \mathbb{M}$ and $\Xi:\mathbb{R}^3\rightarrow \mathbb{N}$ such that
\begin{align}
S(\tau)&=\text{tr}(\tau)I-\tau^T\  \text{for}\  \tau\in \mathbb{M},\label{S defin}\\
\Xi(p)&= \begin{pmatrix} 0 & -p_3 & p_2 \\ p_3 & 0 & -p_1 \\ -p_2 & p_1 & 0 \end{pmatrix} \ \text{for}\  p\in \mathbb{R}^3,
\end{align}
where $I$ denotes the $3 \times 3$ identity matrix.
A direct calculation shows that $S: H(\curl;\Om,\mathbb{M}) \to H(\dvrg;\Om,\mathbb{M})$ and 
\begin{align}
  \curl q : w = -\Xi(\dvr S(q)) : w   \quad \forall q \in H^1(\curl;\Omega,\mathbb{M}), \ \forall w \in L^2(\Om,\mathbb{N}) \ \mbox{ a.e. in } \Omega.     \label{as prop3}
\end{align}

\subsection{Mixed finite element spaces} 
Here we present the MFE approximation
of \eqref{weak-1}--\eqref{weak-3}, which is the basis for the MSMFE methods. We assume that $\Omega$ can be covered by a shape-regular cuboid partition $\Tc_h$ with $h_E = \text{diam} (E)$ for $E \in \Tc_h$ and $h=\max_{E\in \Tc_h} h_E$. For any element $E \in \Tc_h$ there exists a linear
bijection mapping $F_E: \Eh \to E$, where $\Eh$ is the unit reference
cube with vertices $\rh_1 = (0,0,0)^T$, $\rh_2 = (1,0,0)^T$, $\rh_3 = (1,1,0)^T$, $\rh_4 = (0,1,0)^T$,  $\rh_5 = (0,0,1)^T$,  $\rh_6 = (1,0,1)^T$,  $\rh_7 = (1,1,1)^T$, and $\rh_8 = (0,1,1)^T$, with unit outward normal vectors to the edges denoted by $\nh_i$, $i = 1,\ldots,6$, see Figure~\ref{elements}. The mapping $F_E$ is given by
\begin{align}
F_E(\hat \x)  = \r_1 + \r_{21}\xh + \r_{41}\yh + \r_{51}\zh.
\label{mapping}
\end{align}
Denote the Jacobian matrix by $\DF_E$ and let $J_E =
|\!\operatorname{det} (\DF_E)|$. Denote the inverse mapping by $F^{-1}_E$, its Jacobian matrix by $DF^{-1}_E$, and let $J_{F^{-1}_E}=|\text{det}(DF^{-1}_E)|$. For $\x = F_E(\hat\x)$ we have 
$$
\DF^{-1}_E (\x) = (\DF_E)^{-1}(\hat \x), \qquad J_{F^{-1}_E}(\x) = \frac{1}{J_E(\hat\x)}.
$$
The shape-regularity of the grids imply that $\forall E \in \Tc_h$,
\begin{align}
\| \DF_E \|_{0,\infty, \Eh} \sim h_E, \quad \| \DF^{-1}_E \|_{0,\infty, E} \sim h_E^{-1}, \quad \| J_E \|_{0,\infty, \Eh} \sim h_E^3 \quad \text{and} \quad \| J_{F_E^{-1}} \|_{0,\infty, E} \sim h_E^{-3}, \label{scaling-of-mapping}
\end{align}
where the notation $a\sim b$ means that there exist positive constants
$c_0,\, c_1$ independent of $h$ such that $c_0 b \le a \le c_1 b$.

Following the construction on \cite{high-order-mfmfe}, we define the lowest order enhanced Raviart-Thomas space as follows. The $\mathcal{RT}_0$ space is defined on reference cube $\hat{E}$ as
\begin{align}
\mathcal{RT}_0(\hat{E})=\begin{pmatrix} \alpha_1 + \beta_1 \hat{x} \\ \alpha_2 + \beta_2 \hat{y} \\ \alpha_3 + \beta_3 \hat{z} \label{RT0space} \end{pmatrix}, \quad \alpha_i, \beta_i \in \R, \ i = 1,2,3.
\end{align}
Next, define $\mathbf{\mathcal{B}}(\hat{E})=\bigcup_{i=1}^3 \mathbf{\mathcal{B}}_i$, where
\begin{align*}
  \mathbf{\mathcal{B}}_1(\hat{E})&=\text{span}\left\{
  \begin{pmatrix} \hat{y} \\ 0 \\ 0 \end{pmatrix}, 
  \begin{pmatrix} \hat{z} \\ 0 \\ 0 \end{pmatrix}, 
  \begin{pmatrix} \hat{x}\hat{y} \\ 0 \\ 0 \end{pmatrix}, 
  \begin{pmatrix} \hat{x}\hat{z} \\ 0 \\ 0 \end{pmatrix}, 
  \begin{pmatrix} \hat{y}\hat{z} \\ 0 \\ 0 \end{pmatrix}, 
  \begin{pmatrix} \hat{x}\hat{y}\hat{z} \\ 0 \\ 0 \end{pmatrix}
  \right\},\\
  \mathbf{\mathcal{B}}_2(\hat{E})&=\text{span}\left\{
  \begin{pmatrix} 0 \\ \hat{x} \\ 0 \end{pmatrix},
  \begin{pmatrix} 0 \\ \hat{z} \\ 0 \end{pmatrix},
  \begin{pmatrix} 0 \\ \hat{x}\hat{y} \\ 0 \end{pmatrix},
  \begin{pmatrix} 0 \\ \hat{x}\hat{z} \\ 0 \end{pmatrix},
  \begin{pmatrix} 0 \\ \hat{y}\hat{z} \\ 0 \end{pmatrix},
  \begin{pmatrix} 0 \\ \hat{x}\hat{y}\hat{z} \\ 0 \end{pmatrix}
  \right\}, \\ 
  \mathbf{\mathcal{B}}_3(\hat{E})& = \text{span}\left\{
  \begin{pmatrix} 0 \\ 0 \\ \hat{x}  \end{pmatrix},
  \begin{pmatrix} 0 \\ 0 \\ \hat{y}  \end{pmatrix},
  \begin{pmatrix} 0 \\ 0 \\ \hat{x}\hat{y}  \end{pmatrix},
  \begin{pmatrix} 0 \\ 0 \\ \hat{x}\hat{z}  \end{pmatrix},
  \begin{pmatrix} 0 \\ 0 \\ \hat{y}\hat{z}  \end{pmatrix},
  \begin{pmatrix} 0 \\ 0 \\ \hat{x}\hat{y}\hat{z}  \end{pmatrix}
  \right\},
\end{align*}
and consider
\begin{align}
\mathbf{\hat{x}} \times \mathbf{\mathcal{B}}(\hat{E})=&\text{span} \left\{\begin{array}{@{} c @{}}        \begin{pmatrix} 0 \\ \hat{y}\hat{z} \\ -\hat{y}^2 \end{pmatrix} , 
 \begin{pmatrix} 0 \\ \hat{z}^2 \\ -\hat{y} \hat{z} \end{pmatrix},
  \begin{pmatrix} 0 \\ \hat{y}\hat{z}^2 \\ -\hat{y}^2\hat{z} \end{pmatrix},
   \begin{pmatrix} 0 \\ \hat{x}\hat{y}\hat{z} \\ -\hat{x}\hat{y}^2 \end{pmatrix},
    \begin{pmatrix} 0 \\ \hat{x}\hat{z}^2 \\ -\hat{x} \hat{y} \hat{z} \end{pmatrix},
     \begin{pmatrix} 0 \\ \hat{x}\hat{y}\hat{z}^2 \\ -\hat{x}\hat{y}^2 \hat{z} \end{pmatrix}  
    \end{array}\right\}\nonumber\\
    &+\text{span} \left\{\begin{array}{@{} c @{}}        \begin{pmatrix} -\hat{x}\hat{z} \\ 0 \\ \hat{x}^2 \end{pmatrix} , 
 \begin{pmatrix} -\hat{z}^2 \\ 0 \\ \hat{x} \hat{z} \end{pmatrix},
  \begin{pmatrix} -\hat{x}\hat{z}^2 \\ 0 \\ \hat{x}^2\hat{z} \end{pmatrix},
   \begin{pmatrix} -\hat{x}\hat{y}\hat{z} \\ 0 \\ \hat{x}^2\hat{y} \end{pmatrix},
    \begin{pmatrix} -\hat{y}\hat{z}^2 \\ 0 \\ \hat{x}\hat{y} \hat{z} \end{pmatrix},
     \begin{pmatrix} -\hat{x}\hat{y}\hat{z}^2 \\ 0 \\ \hat{x}^2\hat{y} \hat{z} \end{pmatrix}  
    \end{array}\right\}\nonumber\\
    &+\text{span} \left\{\begin{array}{@{} c @{}}        \begin{pmatrix} \hat{x}\hat{y} \\ -\hat{x}^2 \\ 0  \end{pmatrix} , 
 \begin{pmatrix} \hat{y}^2 \\ -\hat{x}\hat{y} \\ 0 \end{pmatrix},
  \begin{pmatrix} \hat{x}\hat{y}^2 \\ -\hat{x}^2\hat{y} \\ 0 \end{pmatrix},
   \begin{pmatrix} \hat{x}\hat{y}\hat{z} \\ -\hat{x}^2\hat{z} \\ 0 \end{pmatrix},
    \begin{pmatrix} \hat{y}^2\hat{z} \\ -\hat{x}\hat{y}\hat{z} \\ 0 \end{pmatrix},
     \begin{pmatrix} \hat{x}\hat{y}^2\hat{z} \\ -\hat{x}^2\hat{y}\hat{z} \\ 0 \end{pmatrix}  
    \end{array}\right\}.\label{x x B elements}
\end{align}
Now, we define the enhanced Raviart-Thomas space $\ERT_0$ on $\hat{E}$ as 
\begin{align}
\ERT_0(\hat{E})=&\mathcal{RT}_0(\hat{E})+\text{curl}(\mathbf{\hat{x}} \times \mathbf{\mathcal{B}}(\hat{E}))\nonumber\\
=&\text{span} \left\{\begin{array}{@{} c @{}}        \begin{pmatrix} 1 \\ 0 \\ 0 \end{pmatrix} , 
 \begin{pmatrix} \hat{x} \\ 0 \\ 0 \end{pmatrix},
  \begin{pmatrix} 0 \\ 1 \\ 0 \end{pmatrix},
   \begin{pmatrix} 0 \\ \hat{y} \\ 0 \end{pmatrix},
    \begin{pmatrix} 0 \\ 0 \\ 1 \end{pmatrix},
     \begin{pmatrix} 0 \\ 0 \\ \hat{z} \end{pmatrix}  
    \end{array}\right\}\nonumber\\
    &+\text{span} \left\{\begin{array}{@{} c @{}}        \begin{pmatrix} -3\hat{y} \\ 0 \\ 0 \end{pmatrix} , 
 \begin{pmatrix} -3\hat{z} \\ 0 \\ 0 \end{pmatrix},
  \begin{pmatrix} -4\hat{y}\hat{z} \\ 0 \\ 0 \end{pmatrix},
   \begin{pmatrix} -3\hat{x}\hat{y} \\ \hat{y}^2 \\ \hat{y}\hat{z} \end{pmatrix},
    \begin{pmatrix} -3\hat{x}\hat{z} \\ \hat{y}\hat{z} \\ \hat{z}^2 \end{pmatrix},
     \begin{pmatrix} -4\hat{x}\hat{y}\hat{z} \\  \hat{y}^2 \hat{z} \\ \hat{y} \hat{z}^2 \end{pmatrix}  
    \end{array}\right\}\nonumber\\
    &+\text{span} \left\{\begin{array}{@{} c @{}}        \begin{pmatrix} 0 \\ -3\hat{x} \\ 0  \end{pmatrix} , 
 \begin{pmatrix} 0 \\ -3\hat{z} \\ 0 \end{pmatrix},
  \begin{pmatrix} 0 \\ -4\hat{x}\hat{z} \\ 0 \end{pmatrix},
   \begin{pmatrix} \hat{x}^2 \\ -3\hat{x}\hat{y} \\ \hat{x}\hat{z} \end{pmatrix},
    \begin{pmatrix} \hat{x}\hat{z} \\ -3\hat{y}\hat{z} \\ \hat{z}^2 \end{pmatrix},
     \begin{pmatrix} \hat{x}^2\hat{z} \\ -4\hat{x}\hat{y}\hat{z} \\ \hat{x}\hat{z}^2 \end{pmatrix}  
    \end{array}\right\}\nonumber\\
&+\text{span} \left\{\begin{array}{@{} c @{}}        \begin{pmatrix} 0 \\ 0 \\ -3\hat{x}  \end{pmatrix} , 
 \begin{pmatrix} 0 \\ 0 \\ -3\hat{y} \end{pmatrix},
  \begin{pmatrix} 0 \\ 0 \\ -4\hat{x}\hat{y} \end{pmatrix},
   \begin{pmatrix} \hat{x}^2 \\ \hat{x}\hat{y} \\ -3\hat{x}\hat{z} \end{pmatrix},
    \begin{pmatrix} \hat{x}\hat{y} \\ \hat{y}^2 \\ -3\hat{y}\hat{z} \end{pmatrix},
     \begin{pmatrix} \hat{x}^2\hat{y} \\ \hat{x}\hat{y}^2 \\ -4\hat{x}\hat{y}\hat{z} \end{pmatrix}  
    \end{array}\right\}. \label{enhanced-RT}
\end{align}

The elasticity finite element spaces on the reference cube are defined as
\begin{align}
\hat{\X}(\Eh) = \left( \ERT_0(\hat{E})\right)^3, \quad
    \Vh(\Eh) = \Qc_0(\Eh,\R^3), \quad 
    \hat{\W}^k(\Eh) = \Qc_k(\Eh,\mathbb{N}) \mbox{ for } k = 0,1, \label{ref-spaces}
\end{align}
where $\Qc_k$ denotes the space of polynomials of degree at most $k$ in each variable. We note that each row of an element of $\X_h(\Eh)$ is a vector in the enhanced Raviart Thomas space $\ERT_0(\Eh)$. It holds that $\dvr \Xh(\Eh) = \Vh(\Eh)$ and for all $\th \in \Xh
(\Eh)$, $\th\, \hat{n}_{\hat{f}} \in \Qc_1(\hat{f},\R^3)$ on any face $\hat{f}$ of $\Eh$.
It is proven \cite{high-order-mfmfe} that the degrees of freedom of
$\ERT_0(\Eh)$ can be chosen as the values of the normal components
at any four points on each of the six faces $\hat{f} \subset \partial \Eh$. In this
work we choose these points to be the vertices of $\hat{f}$, see Figure
\ref{elements}. This is motivated by the vertex quadrature rule,
introduced in the next section. The spaces on any element $E \in \Tc_h$ 
are defined via the transformations
\begin{align*}
 \tau \overset{\Pc}{\leftrightarrow} \hat{\t} : 
\t^T = \frac{1}{J_E} \DF_E \hat{\t}^T \circ F_E^{-1}, 
\quad
    v \leftrightarrow \hat{v} : v = \hat{v} \circ F_E^{-1}, \quad
    w \leftrightarrow \hat{w} : w = \hat{w} \circ F_E^{-1},
\end{align*}
where $\hat\tau \in \Xh(\Eh)$, $\hat v \in \Vh(\Eh)$, and
$\hat{w} \in \hat\W^k(\Eh)$. Note that the Piola 
transformation (applied row-by-row) is used for $\Xh(\Eh)$. It
satisfies, for all sufficiently smooth $\tau \overset{\Pc}{\leftrightarrow} \hat{\t}$, and $v \leftrightarrow \hat{v}$,
\begin{equation}
(\dvr \t, v)_E = (\dvr \hat{\t}, \hat{v})_{\Eh} \quad \text{and} \quad \langle \t\, n_f, v \rangle _f 
= \langle \hat{\t}\, \hat{n}_{\hat f}, \hat{v} \rangle _{\hat f}.\label{prop-piola}
\end{equation}
The spaces on $\Tc_h$ are defined by
\begin{align} \label{spaces}
    \X_h &= \{\tau \in \X: \t|_E \overset{\Pc}{\leftrightarrow} \hat{\t},\: \hat{\t} \in \hat{\X}(\Eh) \quad \forall E\in\mathcal{T}_h\}, \nonumber \\
    V_h &= \{v \in V: v|_E \leftrightarrow \hat{v},\: \hat{v} \in \hat{V}(\Eh) \quad \forall E\in\mathcal{T}_h\}, \\
    \W_h^0 & = \{w \in \W: w|_E \leftrightarrow \hat{w},
\: \hat{w} \in \hat{\W}^0(\Eh) \quad \forall E\in\mathcal{T}_h\}, \nonumber \\
\W_h^1 &  = \{w \in \mathcal{C}(\Om,\N) \subset \W: w|_E \leftrightarrow \hat{w},\: \hat{w} \in \hat{\W}^1(\Eh) \quad \forall E\in\mathcal{T}_h\}. \nonumber
\end{align}
Note that $\W_h^1 \subset H^1(\Omega)$, since it contains continuous
piecewise $\Qc_1$ functions.

\begin{figure}	

	\setlength{\unitlength}{1.0mm}
\scalebox{.7}{
	\begin{picture}(200,55)(-60,0)
	\thicklines
	\Line(5,5)(35,5)
	\Line(5,5)(5,35)    
	\Line(5,35)(35,35)
	\Line(35,35)(35,5)
	\Line(5,35)(20,50)
	\Line(35,35)(50,50)  
	\Line(20,50)(50,50)
	\Line(50,20)(50,50) 
	\Line(35,5)(50,20)   
	\DashLine(5,5)(20,20)    
	\DashLine(50,20)(20,20)
	\DashLine(20,50)(20,20)     
	
	\put(6,1){$\rh_{1}$}
	\put(36,1.5){$\rh_{2}$}
	\put(36.5,31){$\rh_{6}$}
	\put(5.5,31){$\rh_{5}$}
	\put(20,16.5){$\rh_{4}$}
	\put(51,16){$\rh_{3}$}
	\put(51,51.5){$\rh_{7}$}
	\put(16,52){$\rh_{8}$}
	\put(50,20){\vector(1,0){10}} 
	\put(50,50){\vector(1,0){10}} 
	\put(35,5){\vector(1,0){10}} 
	\put(35,35){\vector(1,0){10}} 
	\put(35,35){\vector(0,1){10}} 
	\put(5,35){\vector(0,1){10}} 
	\put(50,50){\vector(0,1){10}} 
	\put(20,50){\vector(0,1){10}} 
	\put(5,5){\vector(-1,-1){5}} 
	\put(35,5){\vector(-1,-1){5}} 
	\put(5,35){\vector(-1,-1){5}} 
	\put(35,35){\vector(-1,-1){5}} 
	\put(5,5){\circle{2}}
	\put(35,5){\circle{2}}
	\put(5,35){\circle{2}}
	\put(35,35){\circle{2}}
	\put(50,50){\circle{2}}
	\put(20,50){\circle{2}}
	\put(50,20){\circle{2}}
	\put(23,23){\circle*{2}}
	\put(27,23){\circle*{2}}
	\put(31,23){\circle*{2}}

	\put(87,30){\vector(1,0){5}}
	\put(95,29){\text{stress}}
	\put(90,25){\circle*{2}}
	\put(95,24){\text{displacement}}
	\put(90,20){\circle{2}}
	\put(95,19){\text{rotation}}

	\end{picture}
}
	\caption{Degrees of freedom of $\X_h\times V_h\times \W^1_h$.}
	\label{elements}
\end{figure}
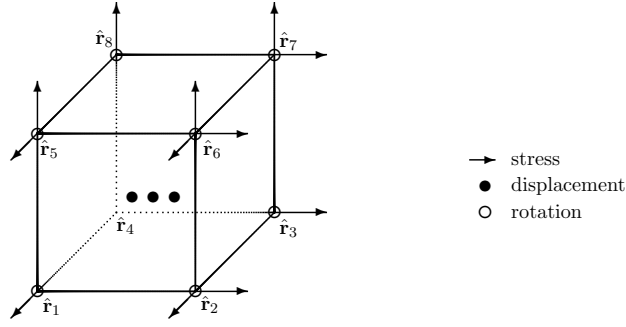

\subsection{A quadrature rule}

Let $\varphi$ and $\psi$ be element-wise continuous functions on $\Omega$. We
denote by $(\varphi,\psi)_Q$ the application of the element-wise vertex quadrature rule for computing $(\varphi,\psi)$.
The integration on any element $E$ is performed by mapping to the
reference element $\Eh$. For $\tau, \, \chi \in \X_h$, we have
\begin{align*}
    \int_{E} A\tau : \chi \, d\x
= \int_{\Eh} \Ah\, \frac{1}{J_E} \th \DF_E^T : \frac{1}{J_E}\ch \DF_E^T\, J_E\, d\hat\x 
= \int_{\Eh} \Ah \th \frac{1}{J_E} \DF^T_E : \ch \DF^T_E\, d\hat\x.
\end{align*}
The quadrature rule on an element $E$ is then defined as 
\begin{equation}
(A\tau,\chi)_{Q,E} \equiv
\left(\Ah \th \frac{1}{J_E} \DF^T_E, \ch \DF^T_E\right)_{\hat Q,\hat E}
\equiv \frac{|\Eh|}{8} \sum_{i=1}^{8} 
\Ah(\rh_i) \th(\rh_i) \frac{1}{J_E(\rh_i)} \DF^T_E(\rh_i) 
: \ch(\rh_i) \DF^T_E(\rh_i). \label{quad-rule-def}
\end{equation}
The global quadrature rule is defined as $(A\tau,\chi)_Q \equiv\sum_{E \in \Tc_h} (A\tau, \chi)_{Q,E}$. 
 
We also employ the vertex quadrature rule for the stress-rotation
bilinear forms in the case of trilinear
rotations. For $\tau \in \X_h,\, w \in \W^1_h$, we have
\begin{equation}
(\t, w)_{Q,E} \equiv 
\left(\frac{1}{J_E}\th\DF_E^T,\hat{w}J_E\right)_{\hat{Q},\hat{E}} 
\equiv \frac{|\hat{E}|}{8}\sum_{i=1}^{8}\th(\rh_i)\DF_E^T(\rh_i)
:\hat{w}(\rh_i). \label{def}
\end{equation}

The next lemma shows that the quadrature rule \eqref{quad-rule-def} produces a
coercive bilinear form.
\begin{lemma}\label{coercivity-lemma}
There exist constants $0 < \alpha_0 \le \alpha_1$ independent of $h$ such that
\begin{equation}\label{coercivity}
\alpha_0 \|\tau\|^2 \le \inp[A\tau]{\tau}_Q \le \alpha_1 \|\tau\|^2 \quad 
\forall \tau\in\X_h.
\end{equation}
Furthermore, $(w,w)^{1/2}_Q$ is a norm in $\W^1_h$ equivalent to
$\|\cdot\|$, and $\forall \, \tau \in \X_h$, $w \in \W^1_h$,
$(\tau,w)_Q \leq C \|\tau\|\|w\|$.
\end{lemma}
\begin{proof}
The proof follows from the argument in \cite[Lemma~2.2]{msmfe-simpl}
\end{proof}

\section{The multipoint stress mixed finite element method with constant 
rotations (MSMFE-0)}

Our first method, referred to as MSMFE-0, is: 
find $\sigma_h \in \X_h,\, u_h \in V_h$, and $\g_h \in \W_h^0$ such that
\begin{align}
(A \s_h,\t)_Q + (u_h,\dvr{\t}) + (\g_h, \t) &= 0, 
& \tau &\in \X_h, \label{h-weak-P0-1}\\
(\dvr \sigma_h,v) &= (f,v), & v &\in V_h, \label{h-weak-P0-2}\\
(\sigma_h,w) &= 0, &w &\in \W_h^0. \label{h-weak-P0-3}
\end{align}

\begin{theorem}
The MSFMFE-0 method \eqref{h-weak-P0-1}--\eqref{h-weak-P0-3} has a unique solution.
\end{theorem}
\begin{proof}
Using classical stability theory of mixed finite element methods,
the required Babu\v{s}ka-Brezzi stability conditions \cite{Boffi-Brezzi-Fortin} are:
\begin{enumerate}[label={\bf(S\arabic*)}]
	\item \label{S1-P0} There exists $c_1 > 0$ such that 
	\begin{align}
	c_1\| \tau \|_{\dvrg} \le \inp[A\t]{\t}^{1/2}_{Q}
	\end{align} 
	for $\tau \in \X_h$ satisfying $\inp[\dvr \t]{v} = 0$ and $\inp[\t]{w} = 0$ for all $(v,w) \in V_h\times \W_h^0$.
	\item \label{S2-P0} There exists $c_2 > 0$ such that 
	\begin{align}
	\inf_{0\neq(v,w)\in V_h\times \W_h^0 } \sup_{0\neq \t\in\X_h} \frac{\inp[\dvr\t]{v}+\inp[\t]{w}}{\| \tau \|_{\dvrg} \left( \|v\| + \|w\| \right)}  \ge c_2.  \label{inf-sup-P0}
	\end{align}
\end{enumerate}
Using \eqref{prop-piola} and $\dvrg \Xh(\Eh) = \Vh(\Eh)$, 
the condition $(\dvrg \tau,v)=0,\, \forall v \in V_h$ implies that 
$\dvr \tau =0$. Then \ref{S1-P0} follows from \eqref{coercivity}. The inf-sup condition \ref{S2-P0} follows from the argument in \cite{awanou2013rectangular} since the $\mathcal{BDM}_1$ space used in Section~4 of that paper is contained in our enhanced Raviart-Thomas space $\mathcal{ERT}_0$ given in \eqref{enhanced-RT}.
\end{proof}

\subsection{Reduction to a cell-centered displacement-rotation system}

The algebraic system that arises from 
\eqref{h-weak-P0-1}--\eqref{h-weak-P0-3} is of the form
\begin{equation}\label{sp-matrix}
    \begin{pmatrix} \As & \Au^T & \Ag^T \\ - \Au & 0   & 0 \\ - \Ag & 0   & 0 \end{pmatrix} 
    \begin{pmatrix} \sigma \\ u \\ \g \end{pmatrix} = 
    \begin{pmatrix} 0 \\ - f \\ 0 \end{pmatrix},
\end{equation}
where $(\As)_{ij} = (A\tau_j,\tau_i)_Q$, $(\Au)_{ij} = (\dvr\tau_j,
v_i)$, and $(\Ag)_{ij} = (\tau_j,w_i)$.  The method can be reduced
to solving a cell-centered displacement-rotation system as
follows. Since the stress degrees of freedom are the four normal components per face evaluated at
the vertices, see Figure~\ref{elements}, the basis
functions associated with a vertex are zero at all other vertices. Therefore the
quadrature rule $(A\sigma_h,\tau)_Q$ decouples the degrees of freedom
associated with a vertex from the rest of the degrees of freedom. As a result,
the matrix $\As$ is block-diagonal with $36\times 36$ blocks associated with vertices, see Figure \ref{stress_elements}. Lemma~\ref{coercivity-lemma} implies that the blocks are symmetric and 
positive definite. Therefore the stress $\sigma_h$ can be easily eliminated by solving small local
systems, resulting in the cell-centered displacement-rotation system 
\begin{equation}\label{msmfe0-system}
    \begin{pmatrix} \Au\As^{-1}\Au^T & \Au\As^{-1}\Ag^T \\ 
\Ag\As^{-1}\Au^T & \Ag\As^{-1}\Ag^T \end{pmatrix} 
    \begin{pmatrix} u \\ \g \end{pmatrix} =
    \begin{pmatrix} \tilde{f} \\ \tilde{h} \end{pmatrix}.
\end{equation}
The displacement and rotation stencils for an element $E$ include all 
elements that share a vertex with $E$. The matrix in 
\eqref{msmfe0-system} is symmetric. Furthermore, for any
$\begin{pmatrix} v^T & w^T \end{pmatrix} \neq 0$, 
\begin{equation}\label{spd-matrix}
\begin{pmatrix} v^T & w^T \end{pmatrix}
\begin{pmatrix} \Au\As^{-1}\Au^T & \Au\As^{-1}\Ag^T 
\\ \Ag\As^{-1}\Au^T & \Ag\As^{-1}\Ag^T \end{pmatrix}
\begin{pmatrix} v \\ w \end{pmatrix} 
= (\Au^Tv + \Ag^Tw)^T\As^{-1}(\Au^Tv + \Ag^Tw) > 0,
\end{equation}
due to the inf-sup condition \ref{S2-P0}, which implies that the matrix
is positive definite.

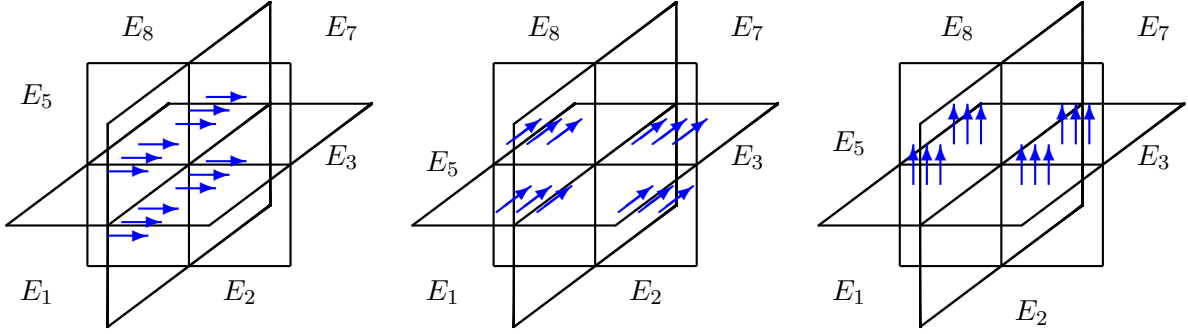
\begin{figure}
    \centering
    \setlength{\unitlength}{.9mm}
    \begin{picture}(160,45)(10,0)
    
        \thicklines
        \put(0,0){
        \put(10,10){\line(1,0){30}}
        \put(10,10){\line(0,1){30}}
        \put(40,10){\line(0,1){30}}
        \put(10,40){\line(1,0){30}}
        \put(10,25){\line(1,0){30}}
        \put(25,10){\line(0,1){30}}
        \put(-2,16){\line(1,0){30}}
        \put(22,34){\line(1,0){30}}
        \Line(22,34)(-2,16)  
        \Line(52,34)(28,16)  
        \Line(37,34)(13,16) 
        \Line(37,49)(13,31) 
        \Line(37,19)(13,1) 
        \Line(37,19)(37,49) 
        \Line(13,1)(13,31) 
        \put(0,5){$E_1$}
        \put(30,5){$E_2$}
        \put(45,25){$E_3$}
        \put(45,44){$E_7$}
        \put(15,44){$E_8$}
        \put(0,34){$E_5$}
        \color{blue}
        \put(17.5,28){\vector(1,0){6}}      
        \put(13,24){\vector(1,0){6}}       
        \put(15,26){\vector(1,0){6}}      
        \put(23,31){\vector(1,0){6}} 	   
        \put(25,33){\vector(1,0){6}} 	 
        \put(27.5,35){\vector(1,0){6}} 	   
        \put(17.5,18.5){\vector(1,0){6}}      
        \put(13,14.5){\vector(1,0){6}}       
        \put(15,16.5){\vector(1,0){6}}      
        \put(23,21.5){\vector(1,0){6}} 	   
        \put(25,23.5){\vector(1,0){6}} 	 
        \put(27.5,25.5){\vector(1,0){6}} 
        \color{black}
        }
        
        \thicklines
        \put(60,0){
        \put(10,10){\line(1,0){30}}
        \put(10,10){\line(0,1){30}}
        \put(40,10){\line(0,1){30}}
        \put(10,40){\line(1,0){30}}
        \put(10,25){\line(1,0){30}}
        \put(25,10){\line(0,1){30}}
        \put(-2,16){\line(1,0){30}}
        \put(22,34){\line(1,0){30}}
        \Line(22,34)(-2,16)  
       	\Line(52,34)(28,16)  
       	\Line(37,34)(13,16) 
        \Line(37,49)(13,31) 
       	\Line(37,19)(13,1) 
       	\Line(37,19)(37,49) 
       	\Line(13,1)(13,31) 
        \put(0,5){$E_1$}
        \put(30,5){$E_2$}
        \put(45,25){$E_3$}
        \put(45,44){$E_7$}
        \put(15,44){$E_8$}
        \put(0,24){$E_5$}
        \color{blue}
        \put(15,28){\vector(4,3){5}}    
        \put(12,28){\vector(4,3){5}}    
        \put(18,28){\vector(4,3){5}}    
        \put(30.5,28){\vector(4,3){5}}   
        \put(33.5,28){\vector(4,3){5}} 
        \put(36.5,28){\vector(4,3){5}} 
        \put(13.5,18){\vector(4,3){5}}    
        \put(10.5,18){\vector(4,3){5}}    
        \put(16.5,18){\vector(4,3){5}}    
        \put(28.5,18){\vector(4,3){5}}   
        \put(31.5,18){\vector(4,3){5}} 
        \put(34.5,18){\vector(4,3){5}} 
        \color{black}
        }
        
        \thicklines
        \put(120,0){
        \put(10,10){\line(1,0){30}}
        \put(10,10){\line(0,1){30}}
        \put(40,10){\line(0,1){30}}
        \put(10,40){\line(1,0){30}}
        \put(10,25){\line(1,0){30}}
        \put(25,10){\line(0,1){30}}
        \put(-2,16){\line(1,0){30}}
        \put(22,34){\line(1,0){30}}
       	\Line(22,34)(-2,16)  
       	\Line(52,34)(28,16)  
       	\Line(37,34)(13,16) 
       	\Line(37,49)(13,31) 
       	\Line(37,19)(13,1) 
       	\Line(37,19)(37,49) 
       	\Line(13,1)(13,31) 
       	\put(0,5){$E_1$}
       	\put(27,2){$E_2$}
    	\put(45,25){$E_3$}
    	\put(45,44){$E_7$}
   	    \put(16,44){$E_8$}
        \put(0,27){$E_5$}
        \color{blue}
        \put(12,22){\vector(0,1){6}}     
        \put(14,22){\vector(0,1){6}}     
        \put(16,22){\vector(0,1){6}}     
        \put(28,22){\vector(0,1){6}}    
        \put(30,22){\vector(0,1){6}}   
        \put(32,22){\vector(0,1){6}}   
        \put(18,28){\vector(0,1){6}}    
        \put(20,28){\vector(0,1){6}}  
        \put(22,28){\vector(0,1){6}}  
        \put(34,28){\vector(0,1){6}}  
        \put(36,28){\vector(0,1){6}}  
        \put(38,28){\vector(0,1){6}}  
        \color{black}
        }
    \end{picture}
    \caption{Interactions of the stress degrees of freedom in the MSMFE methods.}
    \label{stress_elements}
\end{figure}

\begin{remark}
The reduced MSMFE-0 system \eqref{msmfe0-system} is significantly smaller than the original system \eqref{sp-matrix}. To quantify the computational savings, consider a cuboid grid where the number of elements and vertices are approximately equal, denoted by \( m \). In the original system \eqref{sp-matrix}, there are 36 stress degrees of freedom per vertex, three displacement degrees of freedom per element, and three rotation degrees of freedom per element, resulting in approximately \( 42m \) unknowns. The reduced system \eqref{msmfe0-system} has only approximately \( 6m \) unknowns, representing a significant reduction. Moreover, the reduced system is symmetric and positive definite, allowing for the use of efficient solvers such as the conjugate gradient method or multigrid. In contrast, the original system \eqref{sp-matrix} is indefinite, making such fast solvers inapplicable. Additionally, the cost of solving the local vertex systems required to form \eqref{msmfe0-system} is \( O(m) \), which is negligible for large \( m \) compared to the cost of solving the global systems \eqref{sp-matrix} or \eqref{msmfe0-system} using a Krylov space iterative method, which is at least \( O(m^2) \).

Further reduction in the system \eqref{msmfe0-system} is not possible. In the next section, we develop a method involving continuous trilinear rotations and a vertex quadrature rule applied to the stress-rotation bilinear forms. This allows for further local elimination of the rotation, resulting in a cell-centered system for the displacement only.
\end{remark}

\section{The multipoint stress mixed finite element method with trilinear 
rotations (MSMFE-1)}

In the second method, referred to as MSMFE-1, we take $k=1$ in \eqref{ref-spaces}
and apply the quadrature rule to both the stress
bilinear form and the stress-rotation bilinear forms. The method is: find
$\sigma_h \in \X_h,\, u_h \in V_h$ and $\g_h \in \W_h^1$ such that
\begin{align}
(A \s_h,\t)_Q + (u_h,\dvr{\t}) + (\g_h,\t)_Q &= 0, &\tau &\in \X_h, \label{h-weak-P1-1}\\
(\dvr \sigma_h,v) &= (f,v), &v &\in V_h,\label{h-weak-P1-2}\\
(\sigma_h,w)_Q &= 0, &w &\in \W_h^1. \label{h-weak-P1-3}
\end{align}

We note that the rotation finite element space in the MSMFE-1 method is continuous, which may result in reduced approximation if the rotation $\gamma \in L^2(\Om,\mathbb{N})$ is discontinuous. It is possible to consider a modified MSMFE-1 method based on the scaled rotation $\tilde{\gamma}=A^{-1}\gamma$, which is motivated by the relation $\sigma=A^{-1}\nabla u - A^{-1}\gamma$. This method is better suited for problems with discontinuous compliance tensor $A$, since in this case $\sigma$ is smoother than $A\sigma$, implying that $\tilde{\gamma}$ is smoother than $\gamma$. The resulting method is: find $\sigma_h \in \X_h,\, u_h \in V_h$ and $\tilde{\g}_h \in \W_h^1$ such that
\begin{align}
(A \s_h,\t)_Q + (u_h,\dvr{\t}) + (\tilde{\g}_h,A\t)_Q &= 0, &\tau &\in \X_h, \label{h-weak-P1-1-mod}\\
(\dvr \sigma_h,v) &= (f,v), &v &\in V_h,\label{h-weak-P1-2-mod}\\
(A\sigma_h,w)_Q &= 0, &w &\in \W_h^1. \label{h-weak-P1-3-mod}
\end{align}
In the numerical section we present an example with discontinuous $A$ and $\gamma$ illustrating the advantage of the modified method \eqref{h-weak-P1-1-mod}--\eqref{h-weak-P1-3-mod} for problems with discontinuous coefficients. In order to maintain uniformity of the presentation in relation to MSMFE-0, in the following we present the well-posedness and error analysis for the method \eqref{h-weak-P1-1}--\eqref{h-weak-P1-3}. We note that the
analysis for the modified method \eqref{h-weak-P1-1-mod}--\eqref{h-weak-P1-3-mod} is similar.

\subsection{Well-posedness of the MSMFE-1 method}
The classical stability theory of mixed finite element methods \cite{Boffi-Brezzi-Fortin}
implies the following well-posedness result.

\begin{theorem}\label{thm:msmfe-1-cond}
Assume that: 
\begin{enumerate}[label={\bf(S\arabic*)}]
	\setcounter{enumi}{2}
	\item \label{S3-P1}There exists $c_3 > 0$ such that 
	$$ c_3\| \tau \|_{\dvrg}^2 \le \inp[A\t]{\t}_{Q}, $$
	for $\tau \in \X_h$ satisfying $\inp[\dvr \t]{v} = 0$ and $\inp[\t]{w}_{Q} = 0$ for all $(v,w) \in V_h\times \W_h^1$.
	\item \label{S4-P1}There exists $c_4 > 0$ such that 
	\begin{align}
	\inf_{0\neq(v,w)\in V_h\times \W_h^1 } \sup_{0\neq \t\in\X_h} \frac{\inp[\dvr\t]{v}+\inp[\t]{w}_Q}{\| \tau \|_{\dvrg} \left( \|v\| + \|w\| \right)}  \ge c_4. \label{inf-sup-P1}
	\end{align} 
\end{enumerate}
Then the MSMFE-1 method \eqref{h-weak-P1-1}--\eqref{h-weak-P1-3} has a unique solution.
\end{theorem}

The stability condition \ref{S3-P1} holds, since the spaces $\X_h$ and
$V_h$ are as in the MSMFE-0 method. However, 
\ref{S4-P1} is different, due to the quadrature rule in $\inp[\tau]{w}_Q$, 
and it needs to be verified. The next theorem, which is a modification of \cite[Theorem 4.2]{msmfe-simpl}, provides sufficient conditions for a triple $\X_h\times V_h \times \W_h$ to satisfy \ref{S4-P1}.
\begin{theorem}\label{thm:suff-cond}
Suppose that $S_h \subset H(\dvr;\Omega)$ and $U_h \subset L^2(\Omega)$ satisfy
\begin{align}
\inf\limits_{0\neq r\in U_h} \sup\limits_{0\neq z\in S_h} \frac{(\dvr z,r)}{\| z \|_{\dvr} \|r\| }  \ge c_5, \label{darcy-pair}
\end{align}
that $\Q_h \subset H(\curl;\Omega,\mathbb{M})$ and $\W_h \subset L^2(\Omega,\mathbb{N})$ are such 
that $(\cdot,\cdot)_Q^{1/2}$ is a norm in $\W_h$ equivalent to $\|w\|$ and
\begin{align}
\inf\limits_{0\neq w\in \W_h} \sup\limits_{0\neq q\in \Q_h} 
\frac{(\curl  q,w)_Q}{\| \curl q \| \|w\| }  \ge c_6, \label{stokes-pair} 
\end{align}
and that
	\begin{align}
	\curl \Q_h \subset (S_h)^3. \label{curl-condition}
	\end{align}
Then, $\X_h = (S_h)^3 \subset H(\dvrg;\Omega,\M)$, $V_h = (U_h)^3 \subset L^2(\Omega, \mathbb{R}^3)$, and 
$\W_h\subset L^2(\Omega,\mathbb{N})$ satisfy \ref{S4-P1}.
\end{theorem}
\begin{proof}
Let $v\in V_h$ and $w\in \W_h$ be given. It follows from \eqref{darcy-pair} that there exists $\eta\in \X_h$ such that 
\begin{align}
(\dvrg \eta,v)=\|v\|^2,\ \ \ \|\eta\|_{\dvrg}\leq C\|v\|.  \label{stblty-1}
\end{align}
Next, from \eqref{stokes-pair} and \cite[Lemma 2]{arnold2015mixed} there exists $q \in Q_h$ such that 
\begin{align}
P^Q_{W_h}(\curl q)=(w-P^Q_{W_h}(\eta)),\ \ \ \|\curl q\| \leq C(\|w\|+\|\eta\|) \leq C(\|w\|+\|v\|) \label{stblty-Q_h},
\end{align}
where $P^Q_{W_h}:L^2(\Omega,\mathbb{M})\rightarrow \W_h$ is the $L^2$-projection with respect to the norm $(\cdot,\cdot)_Q^{1/2}$. Now let
\begin{equation}
\tau=\eta+\curl q \in \X_h. \label{stblty-tau}
\end{equation}
Using \eqref{stblty-1} and \eqref{stblty-tau}, we have
\begin{align}
(\dvrg \tau,v)=(\dvrg \eta,v)=\|v\|^2. \label{stblty-2}
\end{align}
In addition, \eqref{stblty-1}, \eqref{stblty-Q_h} and \eqref{stblty-tau} imply
\begin{align}
\|\tau\|_{\dvrg}\leq C(\|\eta\|_{\dvrg}+\|\curl q\|)\leq C(\|w\|+\|v\|)\label{stblty-3}
\end{align}
and
\begin{align}
(\tau,w)_Q&=(\eta,w)_Q + (\curl q,w)_Q =(P^Q_{W_h}(\eta),w)_Q + (P^Q_{W_h}(\curl q),w)_Q \nonumber
\\&=(P^Q_{W_h}(\eta),w)_Q + (w-P^Q_{W_h} (\eta),w)_Q =(w,w)_Q\geq C\|w\|^2. \label{stblty-4}
\end{align}
Using \eqref{stblty-2},  \eqref{stblty-3}, and \eqref{stblty-4} we obtain 
\begin{align*}
(\dvrg \tau,v)+(\tau,w)_Q\geq c\|\tau\|_{\dvrg}(\|v\|+\|w\|),
\end{align*}
which completes the proof.
\end{proof}

\begin{remark}
Condition \eqref{darcy-pair} states that $S_h\times U_h$ is a stable Darcy pair.
Condition \eqref{stokes-pair} states that $\Q_h\times \W_h$ is a curl-stable matrix pair with quadrature. The framework in Theorem~\ref{thm:suff-cond} differs from the approach in \cite{awanou2013rectangular}. In particular, the argument in \cite{awanou2013rectangular} requires an additional $H(\dvr;\Omega,\M)$-conforming finite element space $R_h$, which forms a stable Darcy pair with $\W_h$ and can be controlled by $\Q_h$. In \cite{awanou2013rectangular}, $\W_h = \W_h^0$ and $R_h$ is chosen to be the $\RT_0$ space. This approach does not extend to the case $\W_h = \W_h^1$.

The argument in Theorem~\ref{thm:suff-cond} is similar to the approaches in \cite{lee2016towards,arnold2015mixed,msmfe-simpl,msmfe-quads},
where the pair $\Q_h\times \W_h$ is Stokes-stable. The proofs in \cite{arnold2015mixed,msmfe-quads} are for quadrilaterals and do not generalize to three dimensions. In \cite{lee2016towards}, the inf-sup condition \eqref{stokes-pair} is formulated using $(\Xi(\dvr S(q)),w)$,
cf. \eqref{as prop3}. This approach is beneficial if all components of $\Q_h$ belong to the same space, in which case only a vector-vector div-based inf-sup condition needs to be verified. However, this is not true in our case and it is more natural to consider directly a curl-based matrix-matrix inf-sup condition.
\end{remark}

In the following we will establish that the spaces $\X_h \times V_h \times \W_h^1$ in the MSMFE-1 method \eqref{h-weak-P1-1}--\eqref{h-weak-P1-3} satisfy Theorem~\ref{thm:suff-cond}.
A key component of the analysis is defining the space $\Q_h$ satisfying \eqref{curl-condition} and \eqref{stokes-pair}. The space $\Q_h$ is typically chosen to be $H^1$-conforming \cite{lee2016towards}. However, $H^1$-conforming construction for $\Q_h$ in our case is not possible and we construct $H(\curl)$-conforming space $\Q_h$. There are three guiding principles in the construction of each row of $\Q_h$, which we denote by $\Theta_h$. One is the condition $\curl \hat\theta \subset \ERT_0(\hat E)$. The second is that $\hat\theta\times \hat n$ on any face $\hat f$ of $\hat E$ is uniquely determined by the degrees of freedom on $\hat f$, which is needed for $H(\curl)$-conformity. The third is that the degrees of freedom are associated with vertices and edges in a way that \eqref{stokes-pair} holds. We start with the vectors with $\curl$ in $\RT_0(\hat E)$, which are
\begin{equation}\label{curl-in-rt0}
\text{span}\{1,\hat x,\hat y,\hat z, \hat y \hat z\}\times \text{span}\{1,\hat x,\hat y,\hat z, \hat x \hat z\} \times \text{span}\{1,\hat x,\hat y,\hat z, \hat x \hat y\}.
\end{equation}
Next, the functions $\mathbf{\hat{x}} \times \mathbf{\mathcal{B}}(\hat{E})$ in \eqref{x x B elements} are natural candidates for elements of $\Theta(\hat E)$. However, there are not enough of them and they all have two non-zero components. Each of the first three vectors in each row of \eqref{x x B elements} can be split into two separate vectors, since the original and split vectors have the same $\curl$. For example, $(0,\hat{y}\hat{z},-\hat{y}^2)^T$ can be split into $(0,\hat{y}\hat{z},0)^T$ and $(0,0,\hat{y}^2)^T$. Similarly, the three pairs of vectors that have $\hat x \hat y \hat z$ in the same component can be split into three vectors. For example, $(-\hat{x}\hat{y}\hat{z},0,\hat{x}^2\hat{y})^T$ and
$(\hat{x}\hat{y}\hat{z},-\hat{x}^2\hat{z},0)^T$ can be split into $(\hat{x}\hat{y}\hat{z},0,0)^T$, $(0,\hat{x}^2\hat{z},0)^T$ and $(0,0,\hat{x}^2\hat{y})^T$. The last vectors in each of the three rows of \eqref{x x B elements} cannot be split. Therefore we have obtained 30 vectors from \eqref{x x B elements}. Combined with the vectors in \eqref{curl-in-rt0}, this gives 45 vectors. In order to obtain a space with enough degrees of freedom for $H(\curl)$-conformity, we need three more vectors. They have to be curl-free and they should not change the space of $\hat\theta\times \hat n$ on any face $\hat f$. We choose them to be $\text{grad}\{\hat x^2\hat y^2\hat z,\hat x^2\hat y\hat z^2, \hat x\hat y^2\hat z^2\}$. We arrive at the following space on $\hat{E}$:
\begin{align}
& \Theta(\hat{E})= \text{span}   \left\{\begin{array}{@{} c @{}} \nonumber
    1,\hat{x},\hat{y},\hat{z},\hat{x}\hat{y},\hat{x}\hat{z},\hat{y}\hat{z}, \hat{x}\hat{y}\hat{z},\hat{y}^2,\hat{z}^2,\hat{x}\hat{y}^2,\hat{x}\hat{z}^2,\hat{y}^2\hat{z},\hat{y}\hat{z}^2\\
    1,\hat{x},\hat{y},\hat{z},\hat{x}\hat{y},\hat{x}\hat{z},\hat{y}\hat{z}, \hat{x}\hat{y}\hat{z},\hat{x}^2,\hat{z}^2,\hat{x}^2\hat{y},\hat{y}\hat{z}^2,\hat{x}^2\hat{z},\hat{x}\hat{z}^2 \\
    1,\hat{x},\hat{y},\hat{z},\hat{x}\hat{y},\hat{x}\hat{z},\hat{y}\hat{z}, \hat{x}\hat{y}\hat{z},\hat{x}^2,\hat{y}^2,\hat{x}^2\hat{z},\hat{y}^2\hat{z},\hat{x}^2\hat{y},\hat{x}\hat{y}^2
  \end{array}\right\} \\
  & \quad +\text{span} \left\{\begin{array}{@{} c @{}}        \begin{pmatrix} 0 \\ \hat{x}\hat{y}\hat{z}^2 \\ -\hat{x}\hat{y}^2 \hat{z} \end{pmatrix} , 
 \begin{pmatrix} -\hat{x}\hat{y}\hat{z}^2 \\ 0 \\ \hat{x}^2\hat{y} \hat{z} \end{pmatrix},
  \begin{pmatrix} \hat{x}\hat{y}^2\hat{z} \\ -\hat{x}^2\hat{y}\hat{z} \\ 0 \end{pmatrix},
   \begin{pmatrix} 2\hat{x}\hat{y}^2\hat{z} \\ 2\hat{x}^2\hat{y}\hat{z} \\ \hat{x}^2\hat{y}^2 \end{pmatrix},
    \begin{pmatrix} 2\hat{x}\hat{y}\hat{z}^2 \\ \hat{x}^2\hat{z}^2 \\ 2\hat{x}^2\hat{y} \hat{z} \end{pmatrix},
     \begin{pmatrix} \hat{y}^2\hat{z}^2 \\ 2\hat{x}\hat{y}\hat{z}^2 \\ 2\hat{x}\hat{y}^2 \hat{z} \end{pmatrix}  
    \end{array}\right\}.\label{theta ref vector}
\end{align}
The curls of the elements of $\Theta(\hat{E})$ are
\begin{align}
& \curl \Theta(\hat{E})= \text{span}   \left\{\begin{array}{@{} c @{}} \nonumber
    1,\hat{y},\hat{z},\hat{y}\hat{z}\\
    1,\hat{x},\hat{z},\hat{x}\hat{z} \\
    1,\hat{x},\hat{y},\hat{x}\hat{y}
  \end{array}\right\} \\
  & \quad +\text{span} \left\{\begin{array}{@{} c @{}}        \begin{pmatrix} 0 \\ \hat{y} \\ -\hat{z} \end{pmatrix} , 
 \begin{pmatrix} -\hat{x} \\ 0 \\ \hat{z} \end{pmatrix},
  \begin{pmatrix} \hat{x} \\ -\hat{y} \\ 0 \end{pmatrix},
   \begin{pmatrix} 0 \\ \hat{x}\hat{y} \\ -\hat{x}\hat{z} \end{pmatrix},
    \begin{pmatrix} -\hat{x}\hat{y} \\ 0 \\ \hat{y} \hat{z} \end{pmatrix},
     \begin{pmatrix} \hat{x}\hat{z} \\ -\hat{y}\hat{z} \\ 0  \end{pmatrix},\begin{pmatrix} 0 \\ \hat{y}^2 \\ -2\hat{y}\hat{z}\end{pmatrix},\begin{pmatrix} -\hat{x}^2 \\ 0 \\  2\hat{x}\hat{z}\end{pmatrix} ,\begin{pmatrix} \hat{x}^2 \\ -2\hat{x}\hat{y} \\ 0   \end{pmatrix}
    \end{array}\right\}\nonumber\\
& \quad +\text{span} \left\{\begin{array}{@{} c @{}}        \begin{pmatrix} 0 \\ 2 \hat{y}\hat{z} \\ -\hat{z}^2 \end{pmatrix} , 
 \begin{pmatrix} -2\hat{x}\hat{z} \\ 0 \\ \hat{z}^2 \end{pmatrix},
  \begin{pmatrix} 2\hat{x}\hat{y} \\ -\hat{y}^2 \\ 0 \end{pmatrix},
   \begin{pmatrix} -4\hat{x}\hat{y}\hat{z} \\ \hat{y}^2\hat{z} \\ \hat{y}\hat{z}^2 \end{pmatrix},
    \begin{pmatrix} \hat{x}^2\hat{z} \\ -4\hat{x}\hat{y}\hat{z} \\ \hat{x} \hat{z}^2 \end{pmatrix},
     \begin{pmatrix} \hat{x}^2\hat{y} \\ \hat{x}\hat{y}^2 \\ -4\hat{x}\hat{y}\hat{z}  \end{pmatrix}
    \end{array}\right\}.   
    \label{curl theta ref vector}
\end{align}
It follows from \eqref{curl theta ref vector} and \eqref{enhanced-RT} that
\begin{equation}\label{curl-prop}
\text{curl}\, \Theta(\hat{E}) \subset \hat{\ERT}_0(\hat{E}).
\end{equation}
Moreover, $\curl \Theta(\hat{E})$ includes all basis functions of $\hat{\ERT}_0(\hat{E})$ except for $(x,0,0)^T$, $(0,y,0)^T$, and $(0,0,z)^T$, which implies that
$$
\curl \Theta(\hat{E}) = \big\{p \in \hat{\ERT}_0(\hat{E}): \dvr p = 0\big\},
$$
i.e., $\Theta(\hat{E})$ and $\hat{\ERT}_0(\hat{E})$ form an exact sequence.

We next define suitable degrees of freedom for $\Theta(\hat{E})$ that provide unisolvency and $H(\curl)$-conformity.

\begin{lemma}\label{lem:dof}
For $\hat\theta \in \Theta(\hat{E})$, as defined in \eqref{theta ref vector}, the following degrees of freedom determine $\hat\theta$ uniquely:
\begin{itemize}
\item The $x$-component of $\hat\theta$ at the vertices (8 d.o.f.) and at the midpoints of the edges on the faces $x=0$ and $x=1$ (8 d.o.f.),
\item The $y$-component of $\hat\theta$ at the vertices (8 d.o.f.) and at the midpoints of the edges on the faces $y=0$ and $y=1$ (8 d.o.f.),
\item The $z$-component of $\hat\theta$ at the vertices (8 d.o.f.) and at the midpoints of the edges on the faces $z=0$ and $z=1$ (8 d.o.f.).
\end{itemize}
Furthermore, for each face $\hat{f}$ on $\partial\hat{E}$, if the d.o.f. of $\hat\theta\in \Theta(\hat{E})$ on $\hat{f}$ are zero, then $\hat\theta \times \hat{n}=0$ on $\hat{f}$.  \label{h_curl_conf}
\end{lemma}

\begin{figure}	

	\setlength{\unitlength}{1.0mm}
\scalebox{.7}{
	\begin{picture}(70,45)(-40,-5)
	\thicklines
	\Line(-15,5)(15,5)
	\Line(-15,5)(-15,35)    
	\Line(-15,35)(15,35)
	\Line(15,35)(15,5)
	\Line(-15,35)(0,50)
	\Line(15,35)(30,50)  
	\Line(0,50)(30,50)
	\Line(30,20)(30,50) 
	\Line(15,5)(30,20)   
	\DashLine(-15,5)(0,20)    
	\DashLine(30,20)(0,20)
	\DashLine(0,50)(0,20)     
 
	\put(-19,2.5){$\hat\r_{1}$}
	\put(16,2.5){$\hat\r_{2}$}
	\put(16.5,32.5){$\hat\r_{6}$}
	\put(-20,34.5){$\hat\r_{5}$}
	\put(0,16.5){$\hat\r_{4}$}
	\put(31,17.5){$\hat\r_{3}$}
	\put(31,51.5){$\hat\r_{7}$}
	\put(-5,51){$\hat\r_{8}$}
	
	\color{blue}
	\put(-15,5){\circle*{2}}
    \put(-15,5){\vector(1,0){8}}
    \color{red}
	\put(15,5){\circle*{2}}
	\put(7,5){\vector(1,0){8}}
	\color{blue}
	\put(0,20){\circle{2}}
	\put(0,20){\vector(1,0){8}}
	\color{red}
	\put(30,20){\circle*{2}}
	\put(22,20){\vector(1,0){8}}
	\put(30,50){\circle*{2}}
	\put(22,50){\vector(1,0){8}}
	\put(15,35){\circle*{2}}
	\put(7,35){\vector(1,0){8}}
	\put(15,20){\circle*{2}}
	\put(7,20){\vector(1,0){8}}
	\put(30,35){\circle*{2}}
	\put(22,35){\vector(1,0){8}}
	\color{blue}
	\put(-15,20){\circle*{2}}
	\put(-15,20){\vector(1,0){8}}
	\color{red}
	\put(22.5,12.5){\circle*{2}}
	\put(14.5,12.5){\vector(1,0){8}}
	\put(22.5,42.5){\circle*{2}}
	\put(14.5,42.5){\vector(1,0){8}}
	\color{blue}
	\put(-7.5,42.5){\circle*{2}}
	\put(-7.5,42.5){\vector(1,0){8}}
	\put(0,50){\circle*{2}}
	\put(0,50){\vector(1,0){8}}
	\put(-15,35){\circle*{2}}
	\put(-15,35){\vector(1,0){8}}
	\put(-7.5,12.5){\circle{2}}
	\put(-7.5,12.5){\vector(1,0){8}}
	\put(0,34){\circle{2}}
	\put(0,34){\vector(1,0){8}}
	\color{black}
	
	\Line(65,5)(95,5)
	\Line(65,5)(65,35)    
	\Line(65,35)(95,35)
	\Line(95,35)(95,5)
	\Line(65,35)(80,50)
	\Line(95,35)(110,50)  
	\Line(80,50)(110,50)
	\Line(110,20)(110,50) 
	\Line(95,5)(110,20)   
	\DashLine(65,5)(80,20)    
	\DashLine(110,20)(80,20)
	\DashLine(80,50)(80,20)     
	
	\put(61,2.5){$\hat\r_{1}$}
	\put(96,2.5){$\hat\r_{2}$}
	\put(96.5,32.5){$\hat\r_{6}$}
	\put(60,34.5){$\hat\r_{5}$}
	\put(80,16.5){$\hat\r_{4}$}
	\put(111,17.5){$\hat\r_{3}$}
	\put(112,49){$\hat\r_{7}$}
	\put(75.5,51){$\hat\r_{8}$}
	\color{blue}
	
	\put(65,5){\circle*{2}}
	\put(65,5){\vector(1,1){5}}
	\put(80,5){\circle*{2}}
	\put(80,5){\vector(1,1){5}}
	\put(80,35){\circle*{2}}
	\put(80,35){\vector(1,1){5}}
	\put(95,5){\circle*{2}}
	\put(95,5){\vector(1,1){5}}
	\color{red}
	\put(80,20){\circle{2}}
	\put(75,15){\vector(1,1){5}}
	\put(110,20){\circle*{2}}
	\put(105,15){\vector(1,1){5}}
	\put(110,50){\circle*{2}}
	\put(105,45){\vector(1,1){5}}
	\put(95,50){\circle*{2}}
	\put(90,45){\vector(1,1){5}}
	\color{blue}
	\put(95,35){\circle*{2}}
	\put(95,35){\vector(1,1){5}}
	\put(95,20){\circle*{2}}
	\put(95,20){\vector(1,1){5}}
	\color{red}
	\put(110,35){\circle*{2}}
	\put(105,30){\vector(1,1){5}}
	\color{blue}
	\put(65,20){\circle*{2}}
	\put(65,20){\vector(1,1){5}}
	\color{red}
	\put(80,50){\circle*{2}}
	\put(75,45){\vector(1,1){5}}
	\color{blue}
	\put(65,35){\circle*{2}}
	\put(65,35){\vector(1,1){5}}
	\color{red}
	\put(80,33){\circle{2}}
	\put(75,28){\vector(1,1){5}}
	\put(93,20){\circle{2}}
	\put(88,15){\vector(1,1){5}}
	\color{black}
	
	\Line(145,5)(175,5)
	\Line(145,5)(145,35)    
	\Line(145,35)(175,35)
	\Line(175,35)(175,5)
	\Line(145,35)(160,50)
	\Line(175,35)(190,50)  
	\Line(160,50)(190,50)
	\Line(190,20)(190,50) 
	\Line(175,5)(190,20)   
	\DashLine(145,5)(160,20)    
	\DashLine(190,20)(160,20)
	\DashLine(160,50)(160,20)   
	
	\put(141,2.5){$\hat\r_{1}$}
	\put(176,2.5){$\hat\r_{2}$}
	\put(176.5,32.5){$\hat\r_{6}$}
	\put(140,34.5){$\hat\r_{5}$}
	\put(160,16.5){$\hat\r_{4}$}
	\put(191,17.5){$\hat\r_{3}$}
	\put(191,51){$\hat\r_{7}$}
	\put(155,51){$\hat\r_{8}$}
	\color{blue}
	
	\put(145,5){\circle*{2}}
    \put(145,5){\vector(0,1){8}}
	\put(175,5){\circle*{2}}
	\put(175,5){\vector(0,1){8}}
	\put(160,20){\circle{2}}
	\put(160,20){\vector(0,1){8}}
	\put(190,20){\circle*{2}}
	\put(190,20){\vector(0,1){8}}
	\color{red}
	\put(190,50){\circle*{2}}
	\put(190,42){\vector(0,1){8}}
	\put(175,35){\circle*{2}}
	\put(175,27){\vector(0,1){8}}
	\color{blue}
	\put(182.5,12.5){\circle*{2}}
	\put(182.5,12.5){\vector(0,1){8}}
	\color{red}
	\put(182.5,42.5){\circle*{2}}
	\put(182.5,34.5){\vector(0,1){8}}
	\put(152.5,42.5){\circle*{2}}
	\put(152.5,34.5){\vector(0,1){8}}
	\put(160,50){\circle*{2}}
	\put(160,42){\vector(0,1){8}}
	\put(145,35){\circle*{2}}
	\put(145,27){\vector(0,1){8}}
	\color{blue}
	\put(152.5,12.5){\circle{2}}
	\put(152.5,12.5){\vector(0,1){8}}
	\color{red}
	\put(160,35){\circle*{2}}
	\put(160,27){\vector(0,1){8}}
	\put(175,50){\circle*{2}}
	\put(175,42){\vector(0,1){8}}
	\color{blue}
	\put(174,20){\circle{2}}
	\put(174,20){\vector(0,1){8}}
	\put(160,5){\circle*{2}}
	\put(160,5){\vector(0,1){8}}
	\color{black}

	\put(-5,-5){$\Theta_x(\hat{E})$}
	\put(75,-5){$\Theta_y(\hat{E})$}
	\put(155,-5){$\Theta_z(\hat{E})$}
	\end{picture}
}
	\caption{Degrees of freedom for $\Theta(\hat{E})$.}
	\label{Q_1 Q_2 Q_3 figure}
\end{figure}
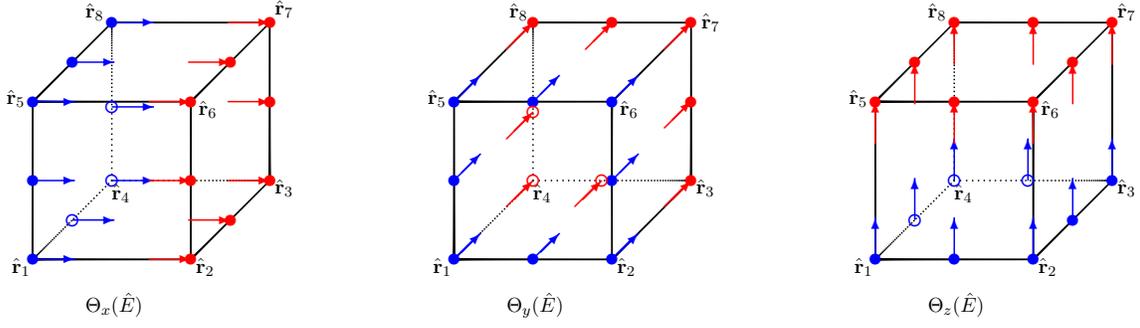

\begin{proof}
The degrees of freedom for the space $\Theta(\hat{E})$ are shown in Figure \ref{Q_1 Q_2 Q_3 figure}. We first prove the second part of the statement of the lemma concerning $H(\curl)$-conformity. Consider a face $\hat f$ with $\hat x = 0$ or $\hat x = 1$ and let the degrees of freedom of $\hat\theta = (\hat\theta_1,\hat\theta_2,\hat\theta_3)^T \in \Theta(\hat{E})$ on this face be zero. The definition \eqref{theta ref vector} implies that the $y$-component has the form
$$
\hat\theta_2 = \alpha_1 + \alpha_2 \hat y + \alpha_3 \hat z + \alpha_4 \hat y \hat z + \alpha_5 \hat z^2 + \alpha_6 \hat y \hat z^2.
$$
On the edge $y=0$, $\hat\theta_2 = \alpha_1 + \alpha_3 \hat z + \alpha_5 \hat z^2$, and $\hat\theta_2 = 0$ at the endpoints and the midpoint of the edge, implying that $\alpha_1 = \alpha_3 = \alpha_5 = 0$. Next, on the edge $y=1$, $\hat\theta_2 = \alpha_2 + \alpha_4 \hat z + \alpha_6 \hat z^2$ and $\hat\theta_2 = 0$ at the endpoints and the midpoint of the edge, implying that $\alpha_2 = \alpha_4 = \alpha_6 = 0$. Therefore $\hat\theta_2 = 0$ on $\hat f$. A similar argument implies that the $z$-component $\hat\theta_3 = 0$ on $\hat f$. Therefore $\hat\theta \times \hat{n}=0$ on $\hat f$. The argument for the other faces is similar. This completes the proof of $H(\curl)$-conformity.

We proceed with the proof of unisolvency. Let all degrees of freedom of $\hat\theta\in \Theta(\hat{E})$ be zero. The above $H(\curl)$-conformity argument implies that $\hat\theta_1 = 0$ for $y = 0,1$ and $z = 0,1$, $\hat\theta_2 = 0$ for $x = 0,1$ and $z = 0,1$, and $\hat\theta_3 = 0$ for $x = 0,1$ and $y = 0,1$. Therefore, 
\begin{align}
&\hat\theta_1 = \yh(\yh-1)\zh(\zh-1)\hat p_1 = (\yh\zh - \yh^2\zh - \yh\zh^2 + \yh^2\zh^2)\hat p_1, \label{uni-1} \\
&\hat\theta_2 = \xh(\xh-1)\zh(\zh-1)\hat p_2 = (\xh\zh - \xh^2\zh - \xh\zh^2 + \xh^2\zh^2)\hat p_2, \label{uni-2} \\
&\hat\theta_3 = \xh(\xh-1)\yh(\yh-1)\hat p_1 = (\xh\yh - \xh^2\yh - \xh\yh^2 + \xh^2\yh^2)\hat p_3, \label{uni-3}
\end{align}
for some polynomials $\hat p_1$, $\hat p_2$, and $\hat p_3$. Since $\hat\theta_1$ includes $\yh^2\zh^2$, but no other terms involving $\yh^2\zh^2$, cf. \eqref{theta ref vector}, we conclude that $\hat p_1$ is a constant. Similarly, $\hat p_2$ and $\hat p_3$ are constants. Let $\hat p_1 = \alpha$ and assume that $\alpha \ne 0$. Then $\hat\theta$ includes $\alpha(\hat{y}^2\hat{z}^2, 2\hat{x}\hat{y}\hat{z}^2, 2\hat{x}\hat{y}^2 \hat{z})^T$. Since $\hat{x}\hat{y}\hat{z}^2$ is not included in the expression of $\hat\theta_2$ in \eqref{uni-2}, it must be eliminated by a combination with other basis functions. The only other basis function in \eqref{theta ref vector} that has $\hat{x}\hat{y}\hat{z}^2$ in $\hat\theta_2$ is $(0,\hat{x}\hat{y}\hat{z}^2,-\hat{x}\hat{y}^2 \hat{z})^T$, which can be used to obtain $\alpha(\hat{y}^2\hat{z}^2, 0 , 4\hat{x}\hat{y}^2 \hat{z})^T$. However, the term $\hat{x}\hat{y}^2 \hat{z}$ is not included in the expression of $\hat\theta_3$ in \eqref{uni-3} and there are no other basis functions in \eqref{theta ref vector} that have this term in $\hat\theta_3$ in order to eliminate it. This leads to a contradiction, implying that $\alpha = 0$. Similar arguments imply that $\hat p_2 = 0$ and $\hat p_3 = 0$. This results in $\hat\theta = 0$.
\end{proof}
We define a space $\Theta_h$ as
\begin{align}
 \Theta_h &= \{\theta \in H(\curl;\Om): \theta|_E \overset{\mathcal{C}}{\leftrightarrow} \hat{\theta},\: \hat{\theta} \in \hat{\Theta}(E) \quad \forall E\in\mathcal{T}_h\}, \label{theta_h definition}
\end{align}
where we use the covariant transformation defined on any element $E \in \mathcal{T}_h$ as
\begin{align}\label{maps-cov}
 \theta \overset{\mathcal{C}}{\leftrightarrow} \hat{\theta} : 
\theta = DF_E^{-T} \hat{\theta} \circ F_E^{-1}.
\end{align}
The covariant transformation preserves the tangential trace of vectors on the element boundary, up to a scaling factor. In particular, if $\hat t$ is a unit tangent vector on a face $\hat f$ of $\hat E$, then $t = DF_E \,\hat t/|DF_E \,\hat t|$ is a unit tangent vector on the corresponding face $f$ of $E$, where $|\cdot|$ denotes the Euclidean norm in $\R^3$. Then
$$
\theta\cdot t = DF_E^{-T} \hat\theta\cdot DF_E \,\hat t/|DF_E\,\hat t| = \hat\theta\cdot\hat t/|DF_E\,\hat t|.
$$
Due to Lemma \ref{h_curl_conf}, $H(\curl)$-conformity can be obtained by matching the degrees of freedom on each face from the two neighboring elements.

A key property is that, for $\theta \in \Theta_h$, $\theta \overset{\mathcal{C}}{\leftrightarrow} \hat{\theta}$, it holds that \cite[(2.1.92)]{Boffi-Brezzi-Fortin}
\begin{align}
 \text{curl} \theta \overset{\mathcal{P}}{\leftrightarrow} \text{curl} \hat{\theta}. \label{covariant property}
\end{align}

\begin{remark}
Property \eqref{covariant property} does not hold in two dimensions when the covariant transformation \eqref{maps-cov} is used. It holds in two dimensions if the standard change of variables is used for $\Theta_h$. This is the approach used \cite{arnold2015mixed,msmfe-quads}, where $\Theta_h$ is $H^1$-conforming.
\end{remark}

We now proceed with showing that the spaces $\X_h \times V_h \times \W_h^1$ in the MSMFE-1 method \eqref{h-weak-P1-1}--\eqref{h-weak-P1-3} satisfy Theorem~\ref{thm:suff-cond}. According to the definition of the spaces in \eqref{ref-spaces} and \eqref{spaces}, we take
\begin{equation}\label{Sh-Uh}
  \begin{split}
& S_h =\{z\in H(\dvrg ;\Omega):z|_E \overset{\Pc}{\leftrightarrow} \hat{z} \in \ERT_0 (\hat{E}), \ z\cdot n=0\  \text{on}\ \Gamma_N \}, \\
    & U_h =\{r \in L^2(\Omega): r|_E \leftrightarrow \hat{r} \in Q_0(\hat{E})\}, \\
    & \W_h = \W_h^1.
\end{split}
\end{equation}
The boundary condition in $S_h$ is needed to guarantee the essential boundary condition in $\X_h$ on $\Gamma_N$. We define the space $\Q_h$ as
\begin{align}\label{Qh-defn}
  \Q_h=\{q\in H(\curl,\Omega,\mathbb{M}): q = (q_1,q_2,q_3)^T, \,
q_i \in \Theta_h,\ i=1,2,3,  \ q=0 \ \text{on}\ \ \Gamma_N\}.
\end{align}

\begin{lemma}\label{lem:two-conditions}
The spaces defined in \eqref{Sh-Uh} and \eqref{Qh-defn} satisfy conditions \eqref{darcy-pair} and \eqref{curl-condition} of Theorem~\ref{thm:suff-cond}.
\end{lemma}

\begin{proof}
  Condition \eqref{darcy-pair} holds, since $\ERT_0 \times Q_0$ is a stable Darcy pair \cite[Lemma 2.1]{high-order-mfmfe}. Next, since for \(q \in \Q_h\) we have \(q = 0\) on \(\Gamma_N\), cf. \eqref{Qh-defn}, it follows that \((\text{curl}\, q_i) \cdot n = 0\) on \(\Gamma_N\) for each row $q_i$, see \cite[Lemma 4.4]{msmfe-simpl}. Then, using \eqref{curl-prop}, \eqref{Sh-Uh}, and \eqref{Qh-defn}, we conclude that $\curl \Q_h \subset (S_h)^3$, implying \eqref{curl-condition}.
\end{proof}

It remains to show that \eqref{stokes-pair} holds. We prove it using several auxiliary lemmas. We make the following assumption:

\begin{enumerate}[label={\bf(A\arabic*)}]
  \item \label{A1} Every element $E \in \mathcal{T}_h$ has at most one face on $\Gamma_N$.
\end{enumerate}

\begin{lemma} \label{no quad glbl lemma 1}
 If \ref{A1} holds, there exists a constant $C$ independent of $h$ such that
for each $w \in \W_h$ there exists $q \in \Q_h$ satisfying
\begin{align}
(\curl q, w)_Q \geq C \sum_{E \in \mathcal{T}_h} \sum_{e \in \partial E} \sum_{l=1}^3 |E| \left( w_l(\r_{j}) - w_l(\r_{k}) \right)^2, \label{no quad glbl-1}
\end{align}
for some constant $C>0$ independent of $h$,
where $\r_{j}$ and $\r_{k}$ are the endpoints of edge $e$.
\end{lemma}

\begin{proof}
  Let $q \in \Q_h$, $q|_E  \overset{\mathcal{C}}{\leftrightarrow} \hat{q} = (\hat q_1, \hat q_2, \hat q_3)^T$, $\hat q_i \in \Theta(\hat E)$. As depicted in Figure \ref{Q_1 Q_2 Q_3 figure}, there are 24 basis functions associated with the midpoints of the edges for each $\hat q_1$, $\hat q_2$, and $\hat q_3$, for a total of 72 such basis functions. Denote by $\displaystyle{\hat q_{i,\alpha,(jk)}}$ the $\alpha$-oriented basis function, $\alpha \in \{x,y,z\}$ associated with $\hat q_i$, $i = 1,2,3$, and the edge connecting the vertices $\hat\r_j$ and $\hat\r_k$. For example, $\hat{q}_{1,x,(14)} = (\hat\theta_{x,(14)},0,0)^T$, where
\begin{equation*}
\hat{\theta}_{x,(14)} = \begin{pmatrix} 4\hat{y} - 4\hat{x}\hat{y} - 4\hat{y}\hat{z} - 4\hat{y}^2 + 4\hat{x}\hat{y}\hat{z} + 4\hat{x}\hat{y}^2 + 4\hat{y}^2\hat{z} - 4\hat{x}\hat{y}^2\hat{z} \\ 0 \\ -\hat{x}\hat{y} + \hat{x}\hat{y}^2 + \hat{x}^2\hat{y} - \hat{x}^2\hat{y}^2 \end{pmatrix}.
\end{equation*}
Consider $\displaystyle{q_{1,x,(14)}|_E \overset{\mathcal{C}}{\leftrightarrow} \hat{q}_{1,x,(14)}}$. Let $w\in W_h$ be given such that
\begin{align*}
w= \begin{pmatrix} 0 & -w_3 & w_2 \\ w_3 & 0 & -w_1 \\ -w_2 & w_1 & 0 \end{pmatrix}.
\end{align*} 
Using \eqref{covariant property}, we obtain by direct calculation:
\begin{align}
  \left(\curl \left(q_{1,x,(14)}\right),w\right)_{Q,E}&=\left(\frac{1}{J_E}\curl \left(\hat{q}_{1,x,(14)}\right)DF^{T}_E,\hat w J_E\right)_{\hat{Q},\hat{E}}\nonumber\\
  &=\left(\curl\big(\hat\theta_{x,(14)},0,0\big)^T DF^{T}_E,\hat w\right)
  _{\hat{Q},\hat{E}}=\frac{h^z_E}{2}\left(w_2(\r_4)-w_2(\r_1)\right). \label{curl bilin ex}
\end{align}

By calculations similar to \eqref{curl bilin ex}, we obtain the following results for the 72 basis functions of $q \in \Q_h$ on the element $E$ associated with the edge midpoints:

\begin{align}
\left(\curl \left(q_{i,\alpha,(jk)}\right), w\right)_{Q,E} &= \frac{h^x_E}{2}\left(w_l(\r_j) - w_l(\r_k)\right) \nonumber \\
\text{for } (i,\alpha,j,k,l) &= (2,y,5,1,3), (2,y,6,2,3), (2,y,7,3,3), (2,y,8,4,3), \nonumber \\
&\quad (2,z,1,4,3), (2,z,2,3,3), (2,z,6,7,3), (2,z,5,8,3), \nonumber \\
&\quad (3,y,1,5,2), (3,y,2,6,2), (3,y,3,7,2), (3,y,4,8,2), \nonumber \\
&\quad (3,z,4,1,2), (3,z,3,2,2), (3,z,7,6,2), (3,z,8,5,2), \label{curl blnr x}
\end{align}
\begin{align}
\left(\curl \left(q_{i,\alpha,(jk)}\right), w\right)_{Q,E} &= \frac{h^y_E}{2}\left(w_l(\r_j) - w_l(\r_k)\right) \nonumber \\
\text{for } (i,\alpha,j,k,l) &= (1,x,5,1,3), (1,x,6,2,3), (1,x,7,3,3), (1,x,8,4,3), \nonumber \\
&\quad (1,z,1,2,3), (1,z,4,3,3), (1,z,8,7,3), (1,z,5,6,3), \nonumber \\
&\quad (3,x,1,5,1), (3,x,2,6,1), (3,x,3,7,1), (3,x,4,8,1), \nonumber \\
&\quad (3,z,2,1,1), (3,z,3,4,1), (3,z,7,8,1), (3,z,6,5,1), \label{curl blnr y}
\end{align}
\begin{align}
\left(\curl \left(q_{i,\alpha,(jk)}\right), w\right)_{Q,E} &= \frac{h^z_E}{2}\left(w_l(\r_j) - w_l(\r_k)\right) \nonumber \\
\text{for } (i,\alpha,j,k,l) &= (1,x,4,1,2), (1,x,3,2,2), (1,x,7,6,2), (1,x,8,5,2), \nonumber \\
&\quad (1,y,1,2,2), (1,y,4,3,2), (1,y,8,7,2), (1,y,5,6,2), \nonumber \\
&\quad (2,x,1,4,1), (2,x,2,3,1), (2,x,6,7,1), (2,x,5,8,1), \nonumber \\
&\quad (2,y,2,1,1), (2,y,3,4,1), (2,y,7,8,1), (2,y,6,5,1), \label{curl blnr z}
\end{align}
\begin{align}
\left(\curl \left(q_{i,\alpha,(jk)}\right), w\right)_{Q,E} &= 0 \nonumber \\
\text{for } (i,\alpha,j,k) &= (1,y,1,5), (1,y,2,6), (1,y,3,7), (1,y,4,8), \nonumber \\
&\quad (1,z,1,4), (1,z,2,3), (1,z,6,7), (1,z,5,8), \nonumber \\
&\quad (2,x,1,5), (2,x,2,6), (2,x,3,7), (2,x,4,8), \nonumber \\
&\quad (2,z,1,2), (2,z,4,3), (2,z,8,7), (2,z,5,6), \nonumber \\
&\quad (3,x,1,4), (3,x,2,3), (3,x,6,7), (3,x,5,8), \nonumber \\
&\quad (3,y,1,2), (3,y,4,3), (3,y,8,7), (3,y,5,6). \label{curl blnr 0}
\end{align}

Let 
\begin{align}
q = \sum_{f \not\subset \Gamma_N}\sum_{e \subset \partial f} |f|(w_l(\r_j) - w_l(\r_k)) q_f^e, \label{no quad glbl varphi}
\end{align}
where $q_f^e$ are the basis functions of $\Q_h$ that appear in \eqref{curl blnr x}, \eqref{curl blnr y}, and \eqref{curl blnr z}. Note that, due to the continuity of $w$, the factor $|f|(w_l(\r_j) - w_l(\r_k))$ is the same for the two elements in the support of $q_f^e$, which are the elements that share $f$. These expressions account for 48 basis functions associated with the midpoints of the edges on an element $E$, and their evaluations include all rotational differences $w_l(\r_j) - w_l(\r_k)$ on all edges $e \in \partial E$. Due to assumption \ref{A1}, if $E$ has a face on $\Gamma_N$, the differences $w_l(\r_j) - w_l(\r_k)$ on an edge $e$ on $\Gamma_N$ can be controlled by the differences on the edges of $E$ that are not on $\Gamma_N$. Using \eqref{curl blnr x}--\eqref{curl blnr z}, the function $q$ defined in \eqref{no quad glbl varphi} satisfies
\begin{align}
(\text{curl} \, q, w)_Q = \sum_{E \in \mathcal{T}_h} (\text{curl} \, q, w)_{Q,E} \geq C \sum_{E \in \mathcal{T}_h} \sum_{e \in \partial E} \sum_{l=1}^3 |E| \left( w_l(\r_{j}) - w_l(\r_{k}) \right)^2,
\end{align}
which is the desired result.
\end{proof}

\begin{remark}
We note that the expression in \eqref{curl bilin ex} without quadrature rule is more involved and the argument in Lemma~\ref{no quad glbl lemma 1} does not apply in this case. Proving stability for the method without quadrature rule for the stress-rotation bilinear forms is an open problem.
\end{remark}

\begin{lemma}\label{lem:discrete-grad}
For all $w \in \W_h$ and all $E \in \Tc_h$, it holds that
\begin{equation}\label{discrete-grad}
  |E|\sum_{e \in \partial E} \sum_{l=1}^3 \left(w_l(\r_j) - w_l(\r_k)\right)^2 \sim h_E^2\|\nabla w\|^2_E.
\end{equation}

\end{lemma}

\begin{proof}
For $w \in \W_h$, at any vertex $\r_j$, each component of $\nabla w$ is of the form $(w_l(\r_j) - w_l(\r_k))/h_{\alpha}$, $\alpha \in \{x,y,z\}$. Then
\begin{equation}\label{discrete-grad-Q}
  |E|\sum_{e \in \partial E} \sum_{l=1}^3 \left(w_l(\r_j) - w_l(\r_k)\right)^2 \sim h_E^2\|\nabla w\|^2_{Q,E}.
\end{equation}
Then \eqref{discrete-grad} follows from \eqref{discrete-grad-Q} and $\|\nabla w\|_{Q,E} \sim \|\nabla w\|_E$, which can be shown in a way similar to the proof of Lemma~\ref{coercivity-lemma}.
\end{proof}

We make the following assumption on the mesh:

\begin{enumerate}[label={\bf(A\arabic*)}]
  \setcounter{enumi}{1}
\item \label{A2} $\Tc_h$ can be agglomerated into a conforming mesh of $k_1 \times k_2 \times k_3$ macroelements $M$, $k_1, k_2, k_3 \in \{2,3\}$.
\end{enumerate}

We need to consider macroelements in order to construct the interpolant in Lemma~\ref{lem:Nedelec} below, which is needed in the proof of Lemma~\ref{lem:Ph}.
Let $\mathbb{P}_h \psi$ denote the $L^2$-projection onto piecewise constants on the macroelement grid. Let $\mathbb{P}^M_h = \mathbb{P}_h|_M$. 

\begin{lemma}\label{no quad glbl main lemma-1}
  If \ref{A1} and \ref{A2} hold, there exist positive constants $C_1$ and $\tilde{C}_1$ independent of $h$ such that for each $w \in \W_h$ there exists $q \in \Q_h$ satisfying
\begin{align}
  (\curl q, w)_Q &\geq C_1 \|w - \mathbb{P}_h w\|^2 \quad \text{and} \quad
  \|\curl q\| \leq \tilde{C}_1\|w - \mathbb{P}_h w\|.\label{no quad glbl-2}
\end{align}
\end{lemma}

\begin{proof}
Let $q$ be defined as in \eqref{no quad glbl varphi}. Then, using \eqref{no quad glbl-1} and \eqref{discrete-grad}, we have
\begin{align}
  (\curl q, w)_Q \geq C \sum_{E \in \mathcal{T}_h} \sum_{e \in \partial E} \sum_{l=1}^3 |E| \left( w_l(\r_{j}) - w_l(\r_{k}) \right)^2
\geq C \sum_{E \in \mathcal{T}_h} h_E^2 \|\nabla w\|^2_E \geq C \sum_{M} h_M^2\|\nabla w\|^2_M, \label{no quad glbl lemma 2 ineq1}
\end{align}
where we used that $w \in H^1(\Omega)$ in the last inequality. By the Friedrichs' inequality on $M$ we get
\begin{align}
\|\nabla w\|_M &= \|\nabla (w - \mathbb{P}^M_h w)\|_M \geq  \frac{C}{h_M} \|w - \mathbb{P}^M_h w\|_M. \label{no quad glbl lemma 2 ineq2}
\end{align}
Then, combining \eqref{no quad glbl lemma 2 ineq1} and \eqref{no quad glbl lemma 2 ineq2} implies
\begin{align*}
(\curl q, w)_Q \geq C \sum_{M} \|w - \mathbb{P}^M_h w\|^2_M,
\end{align*}
which is the first inequality in \eqref{no quad glbl-2}. For the second inequality in \eqref{no quad glbl-2},
we write, for $E \subset M$,
\begin{equation*}
|w_l(\r_{j}) - w_l(\r_{k})| \le |w_l(\r_{j}) - \mathbb{P}^M_h w_l| + |w_l(\r_{k}) - \mathbb{P}^M_h w_l|,
\end{equation*}
which implies, using \eqref{no quad glbl varphi}, $\|q_f^e\|_{0,\infty,E} \sim h_E^{-1}$ (cf. \eqref{maps-cov} and \eqref{scaling-of-mapping}), and an inverse inequality,
\begin{align}
  \|\curl q\|^2_E & \le C h_E^{-2}\|q\|_E^2 \le
  C h_E^{-2}|E| \sum_{e \in \partial E} \sum_{l=1}^3 |f|^2 \left( w_l(\r_{j}) - w_l(\r_{k}) \right)^2 h_E^{-2}
  \nonumber \\
  &
  \le C |E| \sum_{l=1}^3 \sum_{i=1}^8\left( w_l(\r_{i}) - \mathbb{P}^M_h w_l) \right)^2
  \le C \sum_{l=1}^3 \|w_l - \mathbb{P}^M_h w_l\|_{Q,E}^2 \le C \|w - \mathbb{P}^M_h w\|_E^2, \label{w-Pw}
\end{align}
where we used Lemma~\ref{coercivity-lemma} in the last inequality. The second inequality in \eqref{no quad glbl-2} follows from a summation over elements $E$.
\end{proof}

Before we establish control on $\mathbb{P}_h w$, we need an auxiliary result. Consider the lowest order N\'ed\'elec space $\tilde \Theta_h$ \cite{Nedelec}, defined on the reference element $\hat E$ as
$
\tilde \Theta_h(\hat E) = \Pc_{0,1,1} \times \Pc_{1,0,1} \times \Pc_{1,1,0},
$
where $\Pc_{k,l,m}$ denotes the space of polynomials of degree at most $k$, $l$, and $m$ in the corresponding variable. The degrees of freedom of $\tilde \Theta_h(\hat E)$ are given as $\int_{\hat e} \hat\theta\cdot\hat t$ for all edges $\hat e$ of $\hat E$. The space $\tilde \Theta_h$ on $\Omega$ is defined via the covariant transformation \eqref{maps-cov}. It follows from \eqref{theta ref vector} and \eqref{theta_h definition} that $\tilde \Theta_h \subset \Theta_h$.

\begin{lemma}\label{lem:Nedelec}
  There exists an interpolant $\Pi : H^1(\Omega,\R^3) \rightarrow \tilde{\Theta}_h$ such that $\forall \phi \in H^1(\Omega,\R^3)$,
\begin{align}\label{Pi-curl}
  (\curl (\Pi \phi), p_0)_M &= (\curl \phi, p_0)_M \ \forall M, \ \forall p_0 \in \R^3 \text{ on } M,
  \quad \text{and} \quad \|\Pi \phi\|_{\curl} \leq C \|\phi\|_1.
\end{align}
\end{lemma}

\begin{proof}
  Since $(\curl \phi, p_0)_M = \langle \phi \times n_M, p_0 \rangle_{\partial M}$, we need 
  $\langle \Pi \phi \times n_f, p_0 \rangle_{f} =  \langle \phi \times n_f, p_0 \rangle_{f}$ for all faces $f$ of $M$. Consider a vertical face $f$ with $n_f = (1,0,0)^T$. Let $\phi = (\phi_1,\phi_2,\phi_3)^T$; then $\phi \times n_f = (0,\phi_3,-\phi_2)$. Note that there is at least one vertex $\r$ from $\Tc_h$ that is interior to $f$. Let $e_y$ and $e_z$ be $y$- and $z$-oriented edges that meet at $\r$. On the rest of the edges on $f$ we set $\int_e \Pi\phi\cdot t = \int_e \phi\cdot t$. Then we set $(\Pi\phi)_2$ on $e_y$ such that $\int_f(\Pi\phi)_2 = \int_f \phi_2$ and we set $(\Pi\phi)_3$ on $e_z$ such that $\int_f(\Pi\phi)_3 = \int_f \phi_3$. The construction on the rest of the faces of $M$ is similar. This construction satisfies the first property in \eqref{Pi-curl}. The stability bound in \eqref{Pi-curl} follows from a scaling argument.
\end{proof}

\begin{lemma}\label{lem:Ph}
There exists a positive constant $\tilde C_2$ independent of $h$ such that for every $w \in \W_h$ there exists $g \in \Q_h$ satisfying
\begin{align}
(\curl g, \mathbb{P}_h w)_Q &=  \|\mathbb{P}_h w\|^2 \quad \text{and} \quad \|\curl g\| \leq C_2  \|\mathbb{P}_h w\|. \label{no quad glbl-3}
\end{align}
\end{lemma}

\begin{proof}
Let $w \in \W_h$ be arbitrary. There exists $z \in H^1(\Omega,\mathbb{M})$ with $z = 0$ on $\Gamma_N$ such that \cite{Galdi}
\begin{align}\label{div-soln}
-\Xi \, \text{div} \, z &= \mathbb{P}_h w \quad \text{and} \quad \|z\|_1 \leq C \|\mathbb{P}_h w\|.
\end{align}
Let $\tilde \Q_h = \{\tilde q \in H(\curl,\Omega,\mathbb{M}): \tilde q_i = \tilde \theta_i \in \tilde\Theta_h,\ i=1,2,3,  \ \tilde q=0 \ \text{on}\ \ \Gamma_N\}$. It holds that $\tilde \Q_h \subset \Q_h$.
Let $g = \Pi(S^{-1} z) \in \tilde \Q_h$, where $\Pi$ is constructed in Lemma~\ref{lem:Nedelec} and applied row-wise. Using \eqref{Pi-curl} and \eqref{div-soln}, we have
\begin{align}
  \|\curl g\| &\leq \|\Pi(S^{-1} z)\|_{\curl} \leq C \|S^{-1} z\|_1 \leq C \|z\|_1 \leq C \|\mathbb{P}_h w\|.
\end{align}
Using \eqref{Pi-curl}, \eqref{as prop3}, and \eqref{div-soln}, we obtain
\begin{align}
(\curl g, \mathbb{P}_h w) = (\curl (\Pi S^{-1} z), \mathbb{P}_h w) = (\curl (S^{-1} z), \mathbb{P}_h w) = - (\Xi \, (\text{div} \, SS^{-1}z), \mathbb{P}_h w) = \|\mathbb{P}_h w\|^2.
\end{align}
Finally, since $g \in \tilde \Q_h$, the integrated quantity $\curl g: \mathbb{P}_h w$ is linear on $\hat{E}$, implying $(\curl g, \mathbb{P}_h w) = (\curl g, \mathbb{P}_h w)_Q$, which gives the first property in \eqref{no quad glbl-3}.
\end{proof}

\begin{lemma} \label{no quad main lemma}
If \ref{A1} and \ref{A2} hold, there exist positive constants $C_3$ and $\tilde C_3$ independent of $h$ such that for every $w \in \W_h$, there exists $\tilde q \in \Q_h$ satisfying
\begin{align}
(\curl \tilde q, w)_Q &\geq C_3 \|w\|^2, \quad \text{and} \quad \|\curl \tilde q\| \leq \tilde C_3 \|w\|.
\end{align}
\end{lemma}

\begin{proof}
  Let $w \in \W_h$ be given, and consider $q$ and $g \in \Q_h$ satisfying \eqref{no quad glbl-2} and \eqref{no quad glbl-3}, respectively. Set $\tilde q = q + \delta g$, where $\delta = 2C_1(1 + \tilde C^2_2)^{-1}$. We then have
\begin{align*}
(\curl \tilde q, w)_Q &= (\curl q, w)_Q + \delta (\curl g, w)_Q \\
&= (\curl q, w)_Q + \delta (\curl g, \mathbb{P}_h w)_Q + \delta (\curl g, (I - \mathbb{P}_h)w)_Q \\
  & \geq C_1  \|(I - \mathbb{P}_h)w\|^2 + \delta \|\mathbb{P}_h w\|^2 - \delta  \tilde C_2 \|\mathbb{P}_h w\|  \|(I - \mathbb{P}_h)w\| \\
  & \ge C_1  \|(I - \mathbb{P}_h)w\|^2 + \delta \|\mathbb{P}_h w\|^2 - \frac{\delta \tilde C_2^2}{2}\|(I - \mathbb{P}_h)w\|^2 - \frac{\delta}{2}\|\mathbb{P}_h w\|^2 \\
&= C_1 (1 + \tilde C^2_2)^{-1}  \|w\|^2.
\end{align*}
Also, $\|\curl \tilde q\| \leq  \tilde C_1 \|(I - \mathbb{P}_h)w\| + \delta \tilde C_2  \|\mathbb{P}_h w\| \leq C  \|w\|$, completing the proof.
\end{proof}

\begin{theorem}
If \ref{A1} and \ref{A2} hold, the MSMFE-1 method \eqref{h-weak-P1-1}--\eqref{h-weak-P1-3} has a unique solution.
\end{theorem}

\begin{proof}
  Due to Theorem~\ref{thm:msmfe-1-cond} and since \ref{S3-P1} holds, as noted in the paragraph after Theorem~\ref{thm:msmfe-1-cond}, it remains to establish \ref{S4-P1}. This is done using Theorem~\ref{thm:suff-cond} with spaces $S_h$, $U_h$, $\Q_h$, and $\W_h$ defined in \eqref{Sh-Uh} and \eqref{Qh-defn}. With these choices, conditions \eqref{darcy-pair} and \eqref{curl-condition} are established in Lemma~\ref{lem:two-conditions} and the inf-sup condition \eqref{stokes-pair} is established in Lemma~\ref{no quad main lemma}. Since the MSMFE-1 spaces satisfy $\X_h = (S_h)^3$, $V_h = U_h$, and $\W_h^1 = \W_h$, Theorem~\ref{thm:suff-cond} implies that \ref{S4-P1} holds.
\end{proof}

\subsection{Reduction to a cell-centered displacement system of the MSMFE-1 method}
The algebraic system that arises from
\eqref{h-weak-P1-1}--\eqref{h-weak-P1-3} is of the form
\eqref{sp-matrix}, where the matrix $\Ag$ is different from the one in
the MSMFE-0 method, due the the quadrature rule, i.e., $(\Ag)_{ij} =
(\tau_j,w_i)_Q$. As in the MSMFE-0 method, the quadrature rule in
$(A\sigma_h,\tau)_Q$ in \eqref{h-weak-P1-1} localizes the basis
functions interaction around vertices, so the matrix $A_{\s\s}$ is
block diagonal with $36\times 36$ blocks for the cuboid grids.  The stress can be eliminated, resulting
in the displacement-rotation system \eqref{msmfe0-system}. The matrix
in \eqref{msmfe0-system} is symmetric and positive definite, due to
\eqref{spd-matrix} and the inf-sup condition \ref{S4-P1}.

Furthermore, the quadrature rule in the stress-rotation bilinear forms
$(\g_h,\tau)_Q$ and $(\sigma_h,w)_Q$ also localizes the interaction
around vertices, since the rotation basis functions are associated
with the vertices. Therefore the matrix $\Ag$ is block-diagonal with
$3\times 36$ blocks, resulting in a block diagonal rotation 
matrix $\Ag\As^{-1}\Ag^T$ with $3 \times 3$ blocks. As a result, the rotation $\gamma_h$ can be 
easily eliminated from \eqref{msmfe0-system}, leading to the cell-centered
displacement system
\begin{align}
\left( A_{\s u}A^{-1}_{\s\s}A^T_{\s u} - A_{\s u} A^{-1}_{\s\s}A^T_{\s\g}(A_{\s\g}A^{-1}_{\s\s}A^T_{\s\g})^{-1}A_{\s\g}A^{-1}_{\s\s}A^T_{\s u} \right) u = \hat{f} \label{disp-syst}.
\end{align}
The above matrix is symmetric and positive definite, since it is a
Schur complement of the symmetric and positive definite matrix in 
\eqref{msmfe0-system}, see \cite[Theorem 7.7.6]{Horn-Johnson}.

\begin{remark}
The MSMFE-1 method is more efficient than the MSMFE-0 method and
the non-reduced MFE method, since it results in a smaller algebraic system. For example, on a
cuboid grid with approximately $m$ elements and vertices, the MSMFE-1 system
\eqref{disp-syst} has approximately $3m$ unknowns compared to $6m$ unknowns in the MSMFE-0
system \eqref{msmfe0-system} and $42m$ unknowns in the non-reduced MFE system \eqref{sp-matrix}.
\end{remark}
  

\section{Error Estimates}

The error analysis follows closely the analysis developed in \cite[Section~5]{msmfe-quads}. In particular, using the interpolant in $\ERT_0$ developed in \cite{high-order-mfmfe}, it is easy to check that Lemmas~5.1--5.4 and Corollary~5.1 in \cite{msmfe-quads} hold true in our case. Consequently, Theorem~5.1 and Theorem~5.2 also hold true, resulting in the following error estimates, where $Q^u_h$ is the $L^2$-orthogonal projection onto $V_h$ and we denote $A\in W^{j,\infty}_{\Tc_h}$ if $A\in W^{j,\infty}(E)\ \forall E \in \Tc_h$ and $\|A\|_{j,\infty, E}$ is uniformly bounded independently of $h$.

\begin{theorem}
Let $A\in W^{1,\infty}_{\Tc_h}$. If the solution $(\sigma,u,\gamma)$ of \eqref{weak-1}-\eqref{weak-3} is sufficiently smooth, for its numerical approximation $(\sigma_h,u_h,\gamma_h)$ obtained by either MSMFE-0 method \eqref{h-weak-P0-1}-\eqref{h-weak-P0-3} or the MSMFE-1 method \eqref{h-weak-P1-1}-\eqref{h-weak-P1-3}, there exists a constant $C$ independent of $h$ such that
\begin{equation}
\|\sigma-\sigma_h\|_{\dvrg}+\|u-u_h\|+\|\gamma-\gamma_h\|\leq Ch(\|\sigma\|_1+\|\dvrg \sigma\|_1+\|u\|_1+\|\gamma\|_1).
\end{equation}
Furthermore, if $A\in W^{2,\infty}_{\Tc_h}$ and the elasticity problem in $\Omega$ is $H^2$-regular,
there exists a constant $C$ independent of $h$ such that
\begin{equation}
\|Q^u_h u -u_h\|\leq Ch^2(\|\sigma\|_2+\|\gamma\|_2).
\end{equation}
\end{theorem}


\section{Numerical Results}

This section presents numerical experiments to corroborate the theoretical findings discussed in the previous sections. The implementation was carried out using the deal.II finite element library \cite{dealii}. We evaluate the convergence of the MSMFE-0 and MSMFE-1 methods on the cubic grids on $\Omega = (0, 1)^3$. We consider a homogeneous and isotropic medium characterized by
\begin{equation}\label{isotropic}
A\sigma = \frac{1}{2\mu}\left( \sigma -\frac{\lambda}{2\mu + 3\lambda} \text{tr}(\sigma)I \right),
\end{equation}
where $\mu > 0$ and $\lambda > -2\mu/3$ are the Lam\'e coefficients. We address the elasticity problem as formulated in \eqref{elast-1} and \eqref{elast-2}, incorporating Dirichlet boundary conditions. 

\medskip
\noindent
{\bf Example 1: convergence study.} We consider a problem with analytical solution given by
\begin{equation}
u_0 = \begin{pmatrix}
0 \\[8pt]
-(e^x - 1)\left[y - \cos\left(\frac{\pi}{12}\right)(y - \frac{1}{2}) + \sin\left(\frac{\pi}{12}\right)(z - \frac{1}{2}) - \frac{1}{2}\right] \\[8pt]
-(e^x - 1)\left[z - \sin\left(\frac{\pi}{12}\right)(y - \frac{1}{2}) - \cos\left(\frac{\pi}{12}\right)(z - \frac{1}{2}) - \frac{1}{2}\right]
\end{pmatrix}.
\label{ex1-soln}
\end{equation}
The body force is then determined using Lam\'e coefficients \(\lambda = 123\) and \(\mu = 79.3\). The computed solution is shown in Figure~\ref{ex1figure}. In Tables \ref{tab:convergence_table1} and \ref{tab:convergence_table2} we show errors and convergence rates on a sequence of mesh refinements, computed using the MSMFE-0 and MSMFE-1 methods. All rates are in accordance with the error analysis presented in the previous section, including displacement superconvergence.
We note that the MSMFE-1 method with trilinear rotations
exhibits convergence for the rotation of order $O(h^{1.5})$, slightly higher than the theoretical
result.
\begin{table}[H]
\centering
\caption{Relative errors and convergence rates for Example 1 via the MSMFE-0 method.}
\vspace{-0.3cm}
\begin{adjustbox}{max width=\textwidth}
\begin{tabular}{|c|c|c|c|c|c|c|c|c|c|c|}
\specialrule{1pt}{0pt}{0pt}
\rowcolor{blue!20}
 & \multicolumn{2}{c|}{$\|\sigma-\sigma_h\|$} &  \multicolumn{2}{c|}{$\|\text{div}(\sigma-\sigma_h)\|$}&   \multicolumn{2}{c|}{$\|u-u_h\|$}&  \multicolumn{2}{c|}{$\|Q^u_hu-u_h\|$}& \multicolumn{2}{c|}{$\|\gamma-\gamma_h\|$} \\ 
\hline
\rowcolor{blue!20}
$h$ & Error & Rate & Error & Rate & Error & Rate & Error & Rate & Error & Rate \\
\hline
1/2 & $4.397\text{E}-01$ & $-$  & $3.457\text{E}-01$ & $-$ & $6.031\text{E}-01$ & $-$ & $3.838\text{E}-02$ & $-$ & $3.213\text{E}-01$ & $-$ \\
\hline
1/4 & $2.234\text{E}-01$ & $0.98$  & $1.812\text{E}-01$ & $0.93$ & $3.162\text{E}-01$ & $0.93$ & $8.198\text{E}-03$ & $2.23$ & $1.630\text{E}-01$ & $0.98$ \\
\hline
1/8 & $1.121\text{E}-01$ & $0.99$  & $9.208\text{E}-02$ & $0.98$ & $1.600\text{E}-01$ & $0.98$ & $1.977\text{E}-03$ & $2.05$ & $8.180\text{E}-02$ & $0.99$ \\
\hline
1/16 & $5.610\text{E}-02$ & $1.00$  & $4.626\text{E}-02$ & $0.99$ & $8.024\text{E}-02$ & $1.00$ & $4.899\text{E}-04$ & $2.01$ & $4.094\text{E}-02$ & $1.00$ \\
\specialrule{1pt}{0pt}{0pt}
\end{tabular}
\end{adjustbox}
\label{tab:convergence_table1}
\end{table}

\begin{table}[H]
\centering
\caption{Relative errors and convergence rates for Example 1 via the MSMFE-1 method.}
\vspace{-0.3cm}
\begin{adjustbox}{max width=\textwidth}
\begin{tabular}{|c|c|c|c|c|c|c|c|c|c|c|}
\specialrule{1pt}{0pt}{0pt}
\rowcolor{blue!20}
 & \multicolumn{2}{c|}{$\|\sigma-\sigma_h\|$} &  \multicolumn{2}{c|}{$\|\text{div}(\sigma-\sigma_h)\|$}&   \multicolumn{2}{c|}{$\|u-u_h\|$}&  \multicolumn{2}{c|}{$\|Q^u_hu-u_h\|$}& \multicolumn{2}{c|}{$\|\gamma-\gamma_h\|$} \\ 
\hline
\rowcolor{blue!20}
$h$ & Error & Rate & Error & Rate & Error & Rate & Error & Rate & Error & Rate \\
\hline
1/2 & $7.566\text{E}-01$ & $-$  & $3.457\text{E}-01$ & $-$ & $6.034\text{E}-01$ & $-$ & $4.877\text{E}-02$ & $-$ & $3.017\text{E}-01$ & $-$ \\
\hline
1/4 & $4.037\text{E}-01$ & $0.91$  & $1.812\text{E}-01$ & $0.93$ & $3.162\text{E}-01$ & $0.93$ & $1.064\text{E}-02$ & $2.20$ & $1.139\text{E}-01$ & $1.40$ \\
\hline
1/8 & $2.072\text{E}-01$ & $0.96$  & $9.208\text{E}-02$ & $0.98$ & $1.600\text{E}-01$ & $0.98$ & $2.667\text{E}-03$ & $2.00$ & $4.187\text{E}-02$ & $1.44$ \\
\hline
1/16 & $1.047\text{E}-01$ & $0.98$  & $4.626\text{E}-02$ & $0.99$ & $8.024\text{E}-02$ & $1.00$ & $6.829\text{E}-04$ & $1.97$ & $1.511\text{E}-02$ & $1.47$ \\
\specialrule{1pt}{0pt}{0pt}
\end{tabular}
\end{adjustbox}
\label{tab:convergence_table2}
\end{table}

\begin{figure}[ht] 
    \centering
    \begin{subfigure}{.30\textwidth}
        \centering
        \includegraphics[width=\linewidth]{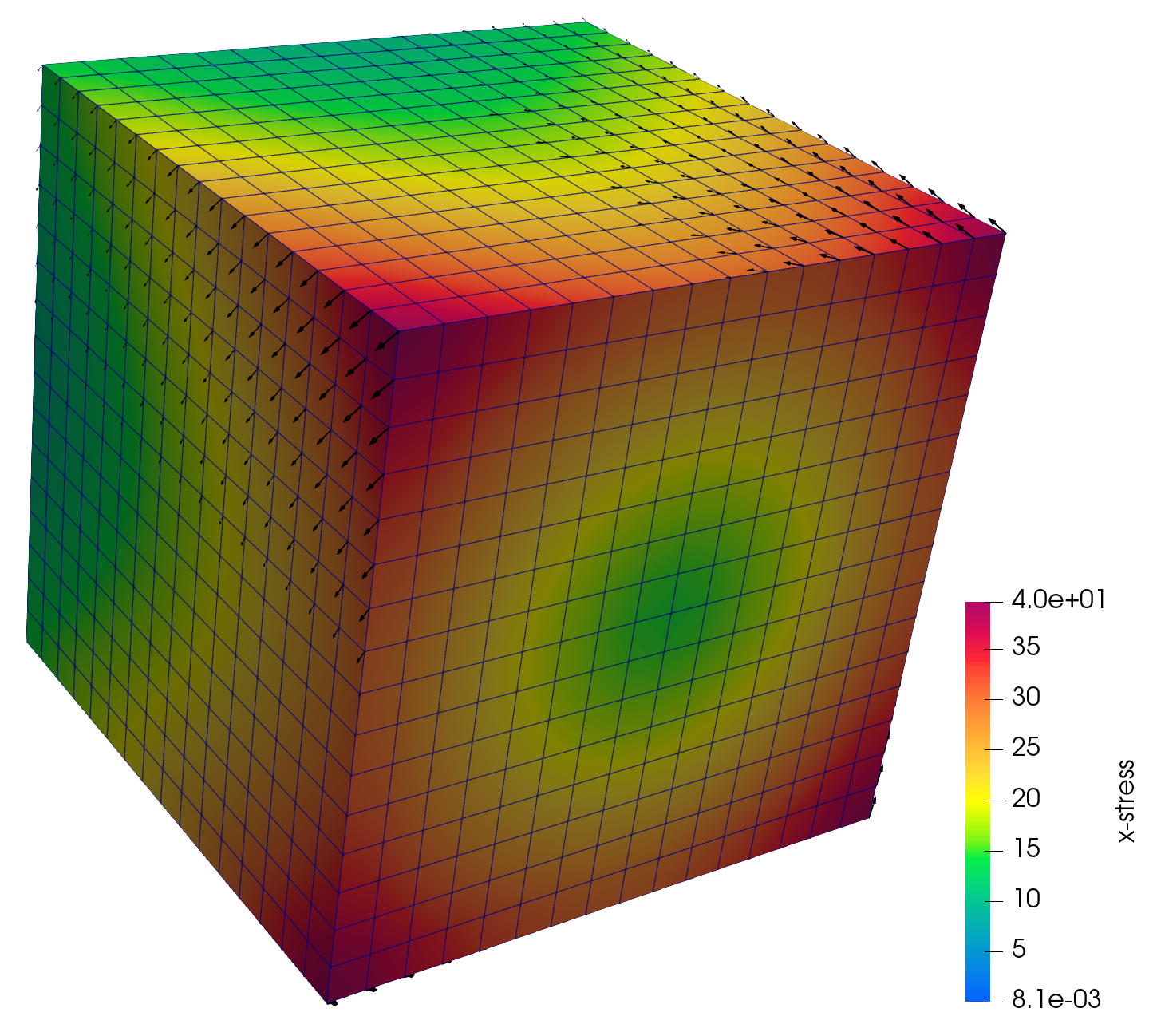}
        \caption{x-stress}
    \end{subfigure}%
    \hspace{3mm}
    \begin{subfigure}{.30\textwidth}
        \centering
        \includegraphics[width=\linewidth]{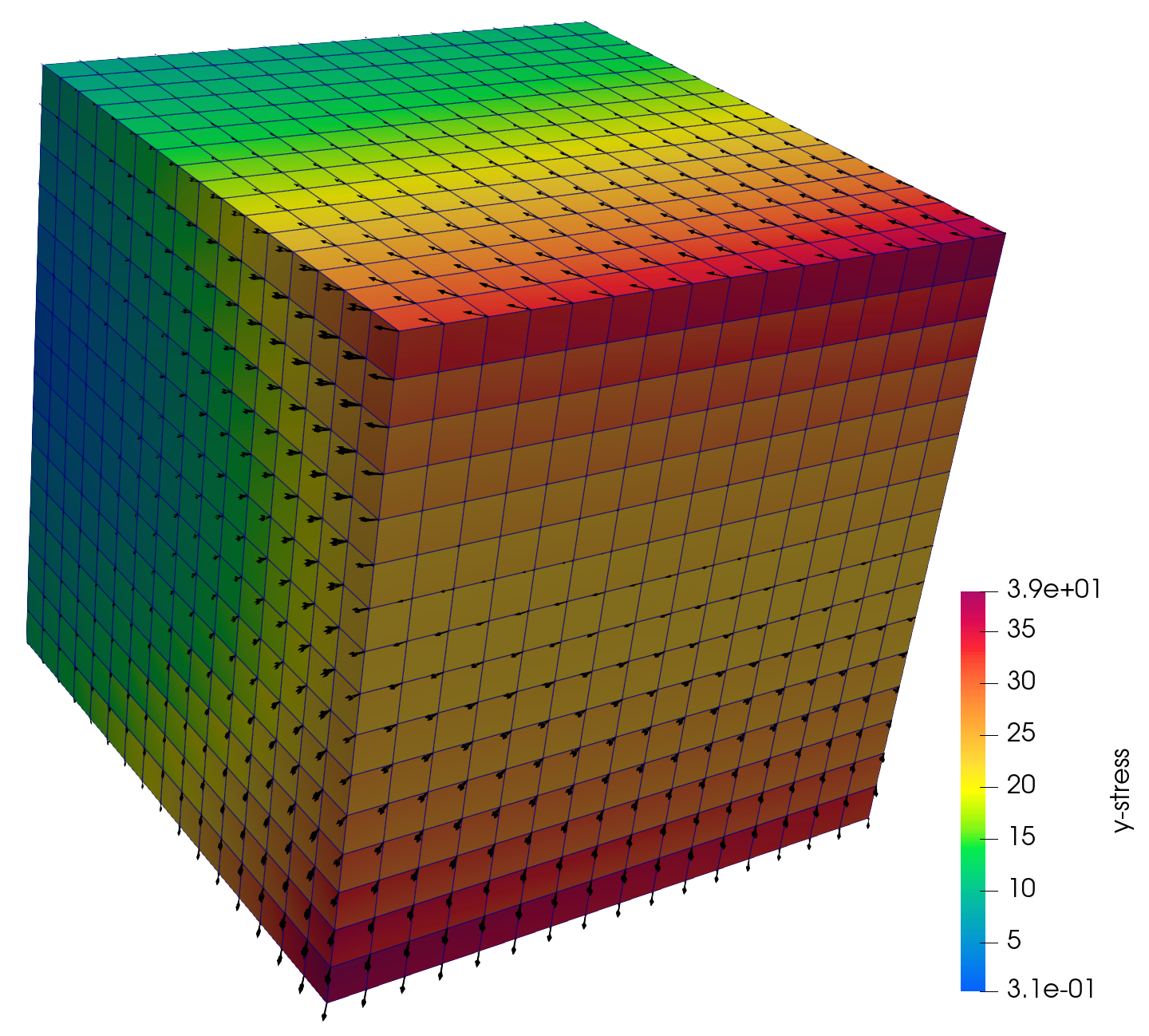}
        \caption{y-stress}
    \end{subfigure}%
        \hspace{3mm}
    \begin{subfigure}{.30\textwidth}
        \centering
        \includegraphics[width=\linewidth]{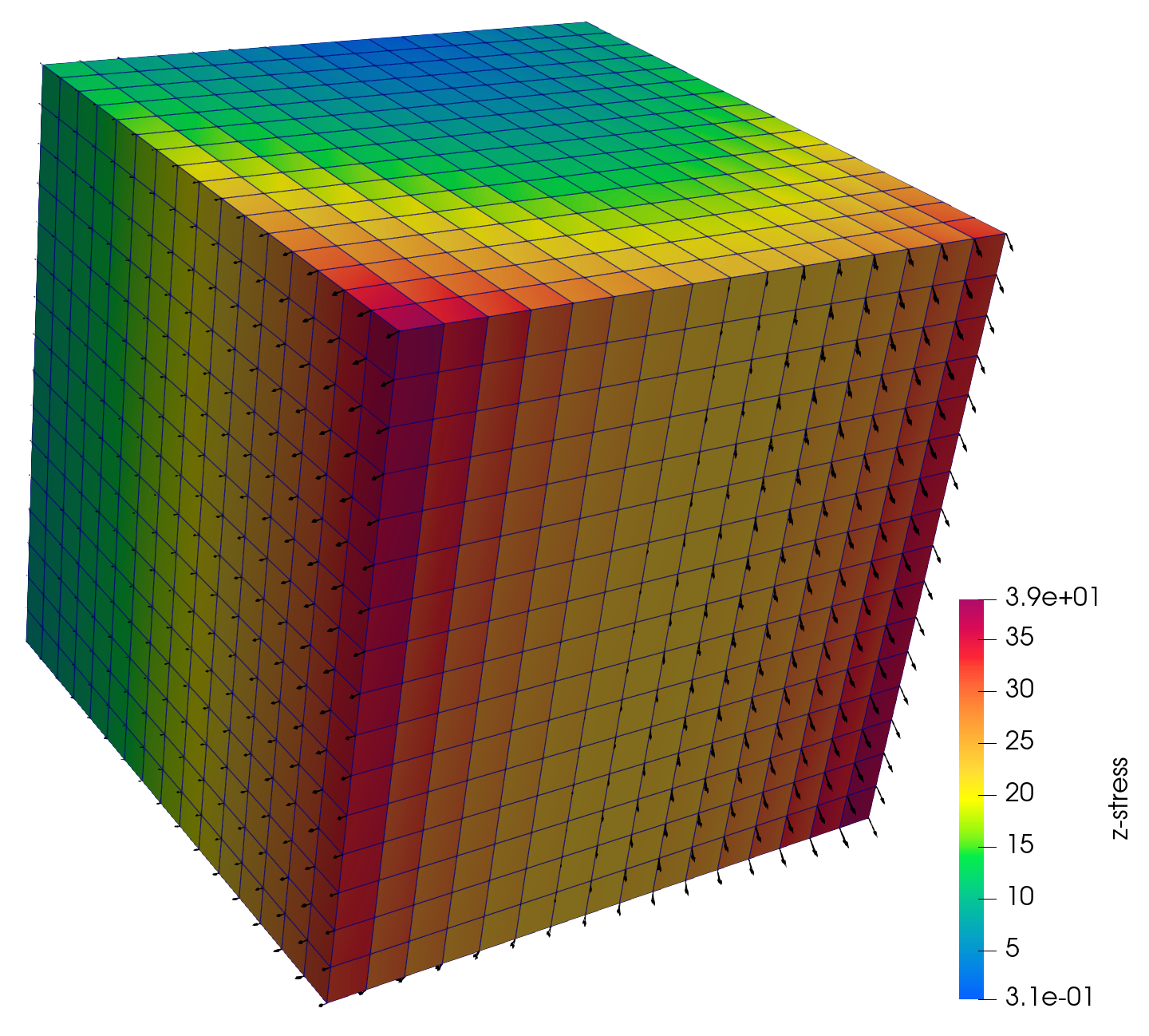}
        \caption{z-stress}
    \end{subfigure}
        \begin{subfigure}{.30\textwidth}
        \centering
        \includegraphics[width=\linewidth]{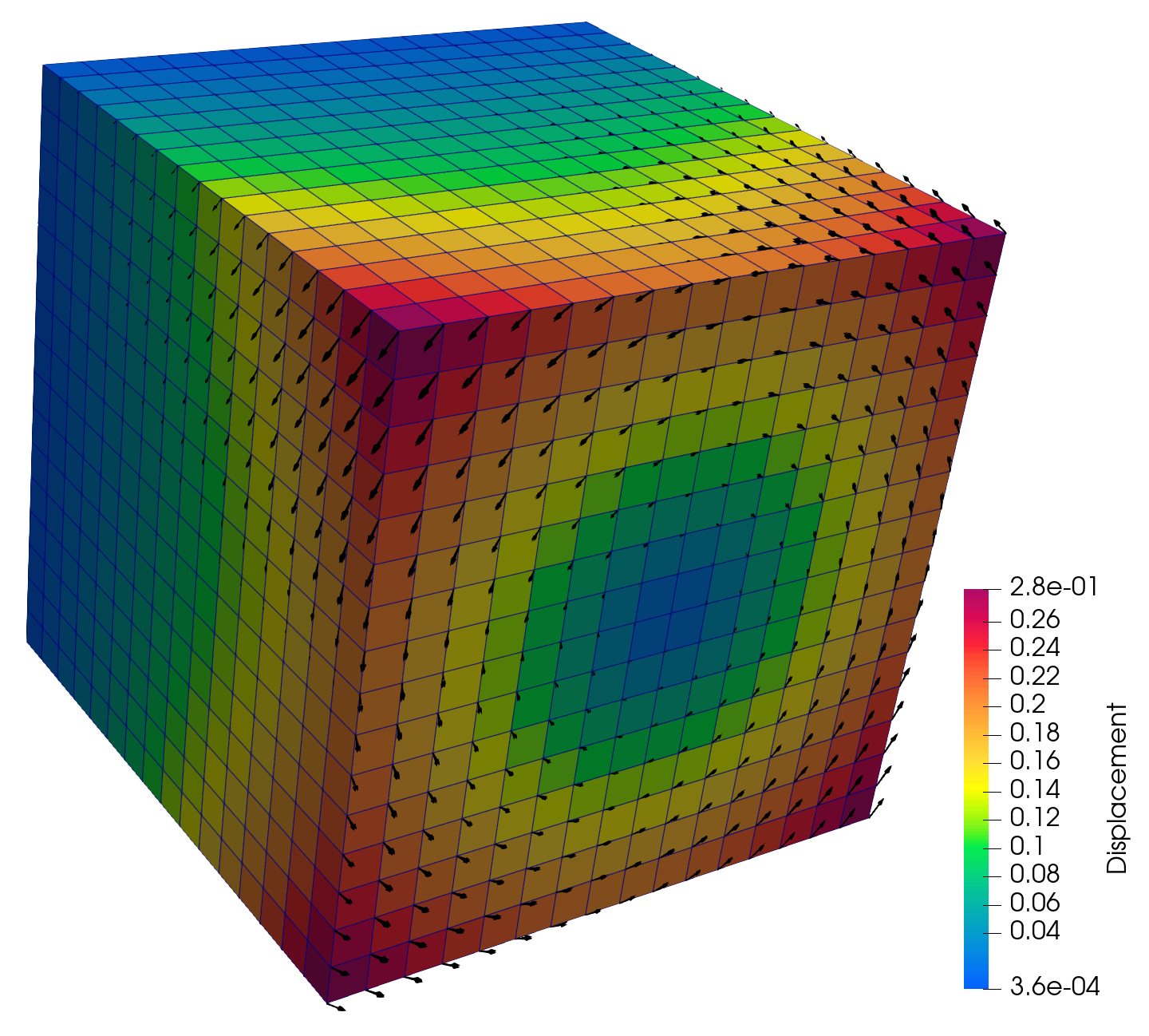}
        \caption{Displacement}
    \end{subfigure}
     \hspace{3mm}
        \begin{subfigure}{.30\textwidth}
        \centering
        \includegraphics[width=\linewidth]{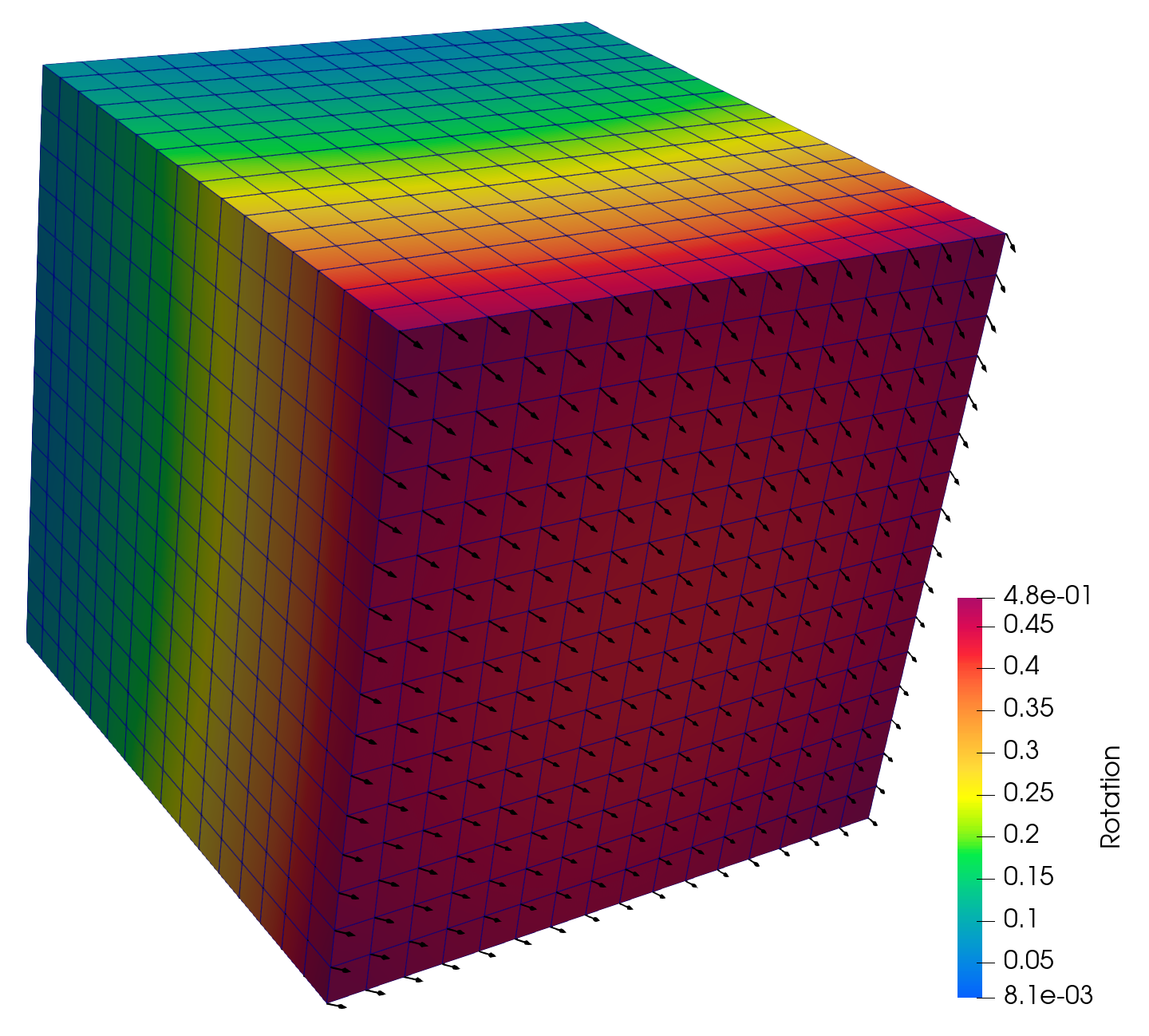}
        \caption{Rotation}
    \end{subfigure}
        \caption{Computed solution for Example 1 via the MSMFE-1 method, $h = 1/16$.}
         \label{ex1figure}
\end{figure}

\noindent
{\bf Example 2: discontinuous coefficient.} This example demonstrates the performance of the MSMFE
methods for discontinuous materials. We consider a $2 \times  2 \times 2$ partitioning of the unit
cube and introduce heterogeneity through
\begin{align*}
\chi(x,y,z) = 
\begin{cases} 
1 & \text{if } \max(x,y,z) < \frac{1}{2}, \\ 
0 & \text{otherwise}.
\end{cases}
\end{align*}
We set $\kappa=10^6$ to characterize the jump in the Lam\'e coefficients and take $\lambda  = \mu  =
(1- \chi ) + \kappa \chi$ . We choose a discontinuous displacement solution as
\begin{align*}
u = \frac{1}{(1 - \chi) + \kappa \chi} \begin{pmatrix} \sin(2\pi x) \sin(2\pi y) \sin(2\pi z) \\ \sin(2\pi x) \sin(2\pi y) \sin(2\pi z) \\ \sin(2\pi x) \sin(2\pi y) \sin(2\pi z) \end{pmatrix},
\end{align*}
so that the stress is continuous and independent of $\kappa$. The body forces are determined
from the above solution using the governing equations. We note that the rotation
$\gamma  = \text{Skew}(\nabla u)$ is discontinuous. The MSMFE-0 method, which has discontinuous displacements
and rotations, handles properly the discontinuity in these variables and
exhibits first order convergence in all variables, as well as displacement superconvergence, see Table \ref{tab:convergence_table3}. The MSMFE-1 method uses continuous rotations
and does not resolve the rotation discontinuity, which results in a reduced convergence
rate for the rotation, as well as the stress. Instead, we can use the modified
MSMFE-1 method \eqref{h-weak-P1-1-mod}--\eqref{h-weak-P1-3-mod} based on the scaled rotation $\tilde{\gamma}  = A^{-1}\gamma$, which in this case is continuous.  The computed solution using the modified MSMFE-1 method, which includes the scaled rotation, is presented in Figure \ref{ex2figure}. The plots display the domain cut along $y=0.25$ to reveal the results within the interior. Table \ref{tab:convergence_table4} indicates that the method exhibits the same
order of convergence for all variables as for smooth problems.

\begin{table}[H]
\centering
\caption{Relative errors and convergence rates for Example 2 via the MSMFE-0 method.}
\vspace{-0.3cm}
\begin{adjustbox}{max width=\textwidth}
\begin{tabular}{|c|c|c|c|c|c|c|c|c|c|c|}
\specialrule{1pt}{0pt}{0pt}
\rowcolor{blue!20}
 & \multicolumn{2}{c|}{$\|\sigma-\sigma_h\|$} &  \multicolumn{2}{c|}{$\|\text{div}(\sigma-\sigma_h)\|$}&   \multicolumn{2}{c|}{$\|u-u_h\|$}&  \multicolumn{2}{c|}{$\|Q^u_hu-u_h\|$}& \multicolumn{2}{c|}{$\|\gamma-\gamma_h\|$} \\ 
\hline
\rowcolor{blue!20}
$h$ & Error & Rate & Error & Rate & Error & Rate & Error & Rate & Error & Rate \\
\hline
1/2 & $1.000\text{E}+00$ & $-$  & $1.000\text{E}+00$ & $-$ & $1.000\text{E}+00$ & $-$ & $1.000\text{E}+00$ & $-$ & $1.000\text{E}+00$ & $-$ \\
\hline
1/4 & $7.662\text{E}-01$ & $0.38$  & $7.867\text{E}-01$ & $0.35$ & $7.626\text{E}-01$ & $0.39$ & $6.023\text{E}-01$ & $0.73$ & $7.662\text{E}-01$ & $0.38$ \\
\hline
1/8 & $3.466\text{E}-01$ & $1.14$  & $4.097\text{E}-01$ & $0.94$ & $3.975\text{E}-01$ & $0.94$ & $1.901\text{E}-01$ & $1.66$ & $3.952\text{E}-01$ & $0.96$ \\
\hline
1/16 & $1.515\text{E}-01$ & $1.19$  & $2.030\text{E}-01$ & $1.01$ & $1.974\text{E}-01$ & $1.01$ & $5.088\text{E}-02$ & $1.90$ & $1.969\text{E}-01$ & $1.00$ \\
\specialrule{1pt}{0pt}{0pt}
\end{tabular}
\end{adjustbox}
\label{tab:convergence_table3}
\end{table}

\begin{table}[H]
\centering
\caption{Relative errors and convergence rates for Example 2 via the modified MSMFE-1 method.}
\vspace{-0.3cm}
\begin{adjustbox}{max width=\textwidth}
\begin{tabular}{|c|c|c|c|c|c|c|c|c|c|c|}
\specialrule{1pt}{0pt}{0pt}
\rowcolor{blue!20}
 & \multicolumn{2}{c|}{$\|\sigma-\sigma_h\|$} &  \multicolumn{2}{c|}{$\|\text{div}(\sigma-\sigma_h)\|$}&   \multicolumn{2}{c|}{$\|u-u_h\|$}&  \multicolumn{2}{c|}{$\|Q^u_hu-u_h\|$}& \multicolumn{2}{c|}{$\|\gamma-\gamma_h\|$} \\ 
\hline
\rowcolor{blue!20}
$h$ & Error & Rate & Error & Rate & Error & Rate & Error & Rate & Error & Rate \\
\hline
1/2 & $1.000\text{E}+00$ & $-$  & $1.000\text{E}+00$ & $-$ & $1.000\text{E}+00$ & $-$ & $1.000\text{E}+00$ & $-$ & $1.000\text{E}+00$ & $-$ \\
\hline
1/4 & $7.797\text{E}-01$ & $0.36$  & $7.867\text{E}-01$ & $0.35$ & $7.806\text{E}-01$ & $0.36$ & $6.388\text{E}-01$ & $0.65$ & $8.836\text{E}-01$ & $0.18$ \\
\hline
1/8 & $3.816\text{E}-01$ & $1.03$  & $4.097\text{E}-01$ & $0.94$ & $4.278\text{E}-01$ & $0.87$ & $2.665\text{E}-01$ & $1.26$ & $5.144\text{E}-01$ & $0.78$ \\
\hline
1/16 & $1.753\text{E}-01$ & $1.12$  & $2.030\text{E}-01$ & $1.01$ & $2.067\text{E}-01$ & $1.05$ & $8.675\text{E}-02$ & $1.62$ & $2.012\text{E}-01$ & $1.35$ \\
\hline
1/32 & $8.371\text{E}-02$ & $1.07$  & $1.011\text{E}-01$ & $1.01$ & $9.993\text{E}-02$ & $1.05$ & $2.404\text{E}-02$ & $1.85$ & $6.594\text{E}-02$ & $1.61$ \\
\specialrule{1pt}{0pt}{0pt}
\end{tabular}
\end{adjustbox}
\label{tab:convergence_table4}
\end{table}

\begin{figure}[ht] 
    \centering
    \begin{subfigure}{.30\textwidth}
        \centering
        \includegraphics[width=\linewidth]{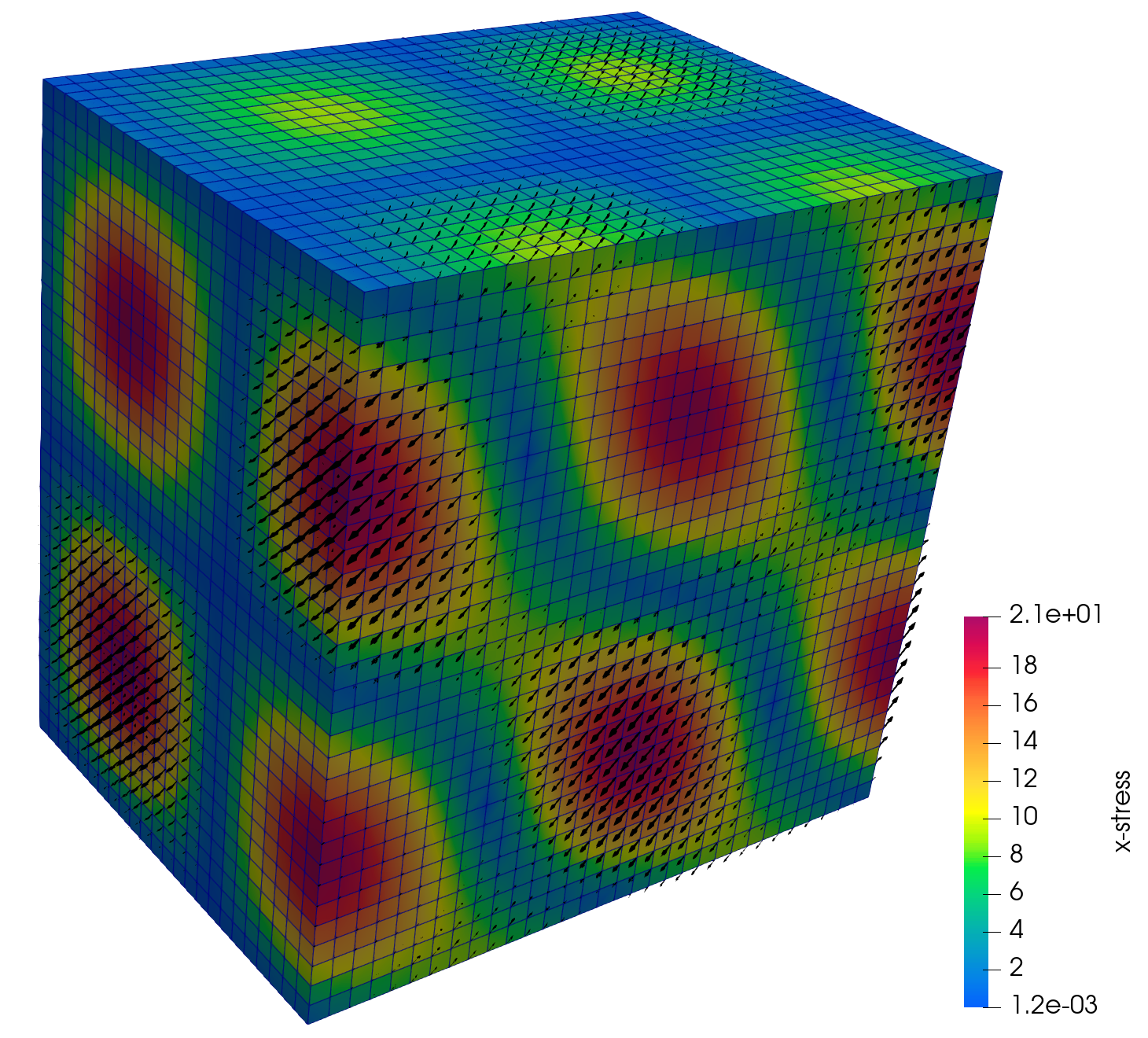}
        \caption{x-stress}
    \end{subfigure}%
    \hspace{3mm}
    \begin{subfigure}{.30\textwidth}
        \centering
        \includegraphics[width=\linewidth]{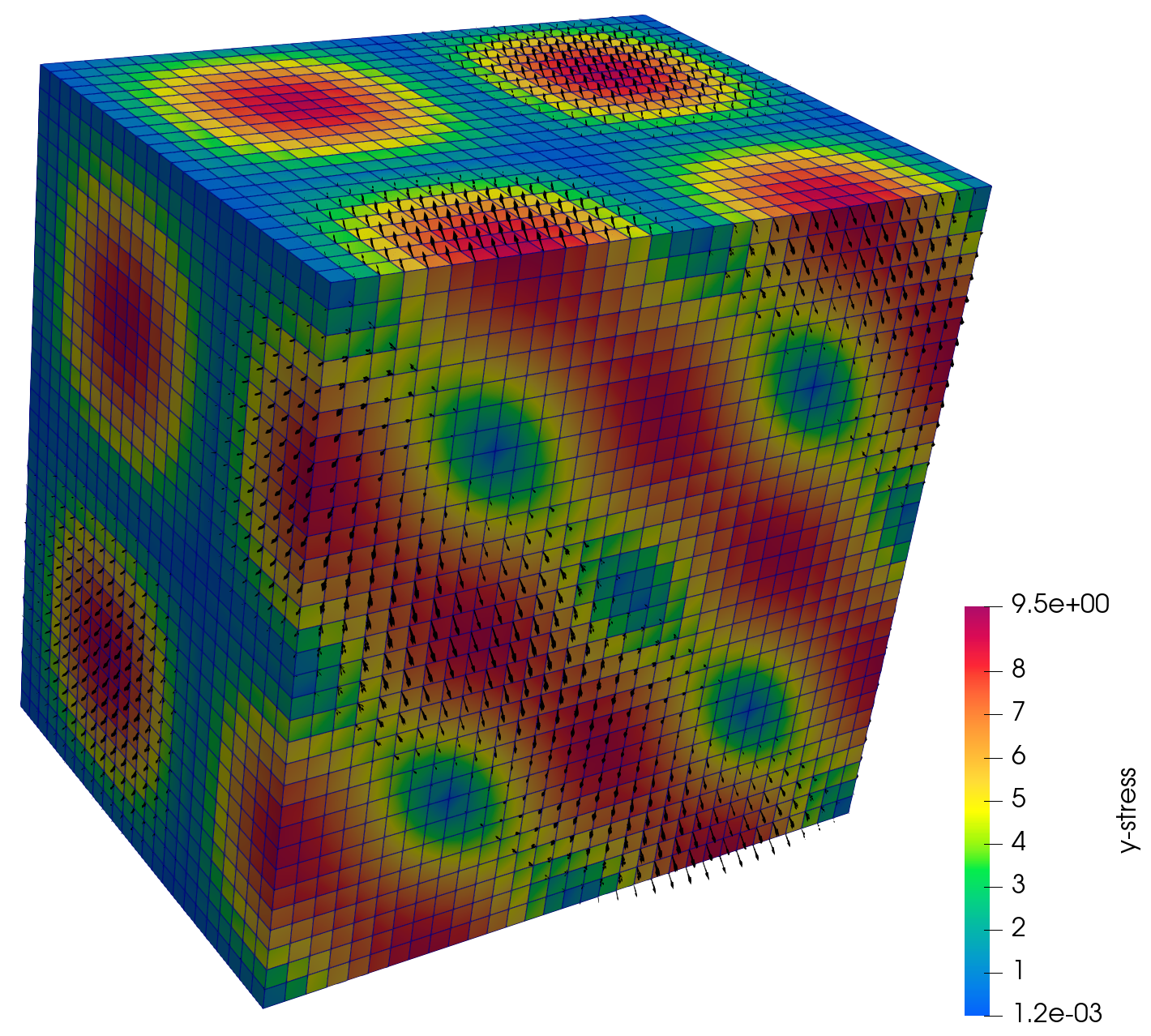}
        \caption{y-stress}
    \end{subfigure}%
        \hspace{3mm}
    \begin{subfigure}{.30\textwidth}
        \centering
        \includegraphics[width=\linewidth]{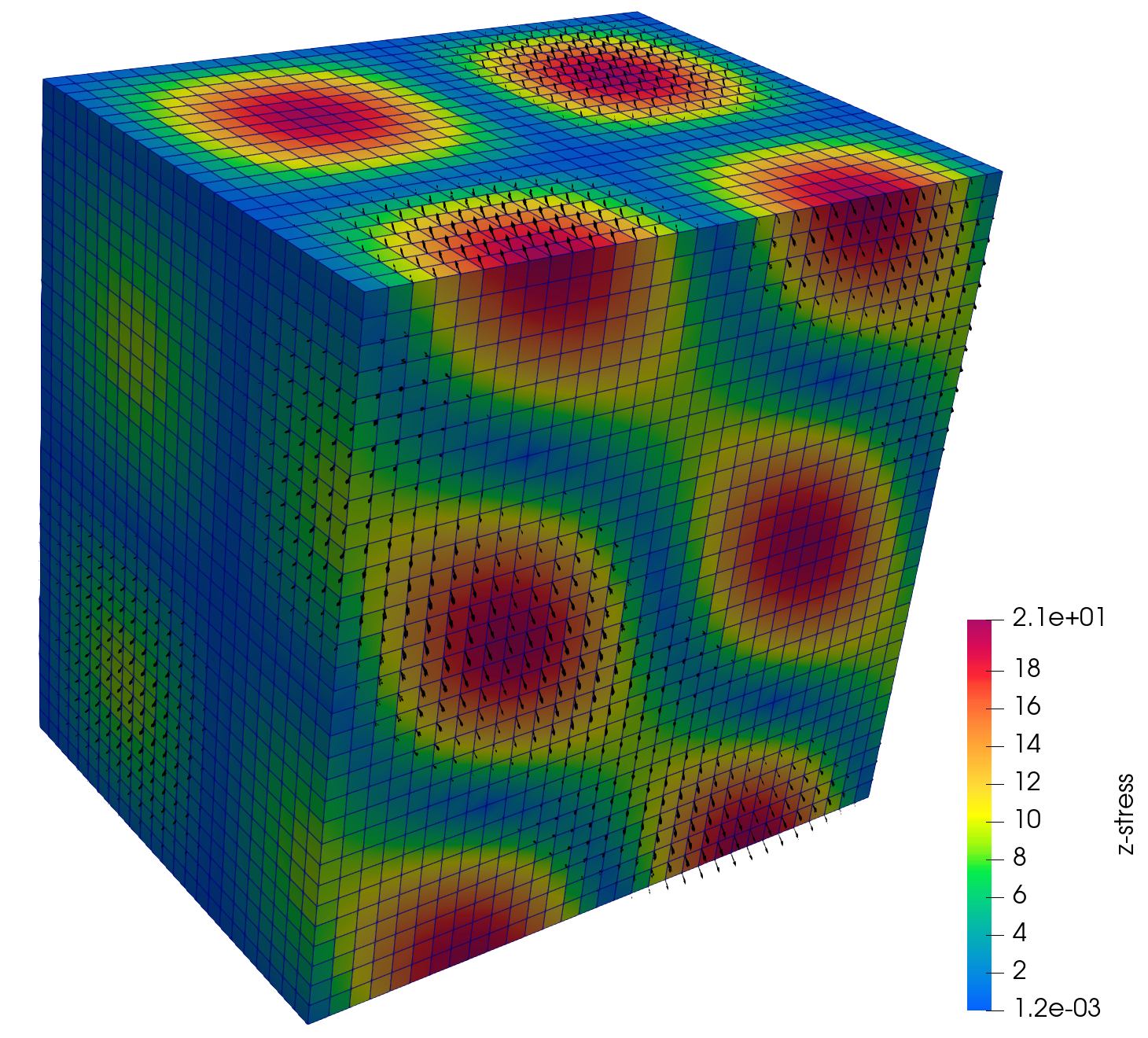}
        \caption{z-stress}
    \end{subfigure}
        \begin{subfigure}{.30\textwidth}
        \centering
        \vspace{.3cm}
        \includegraphics[width=\linewidth]{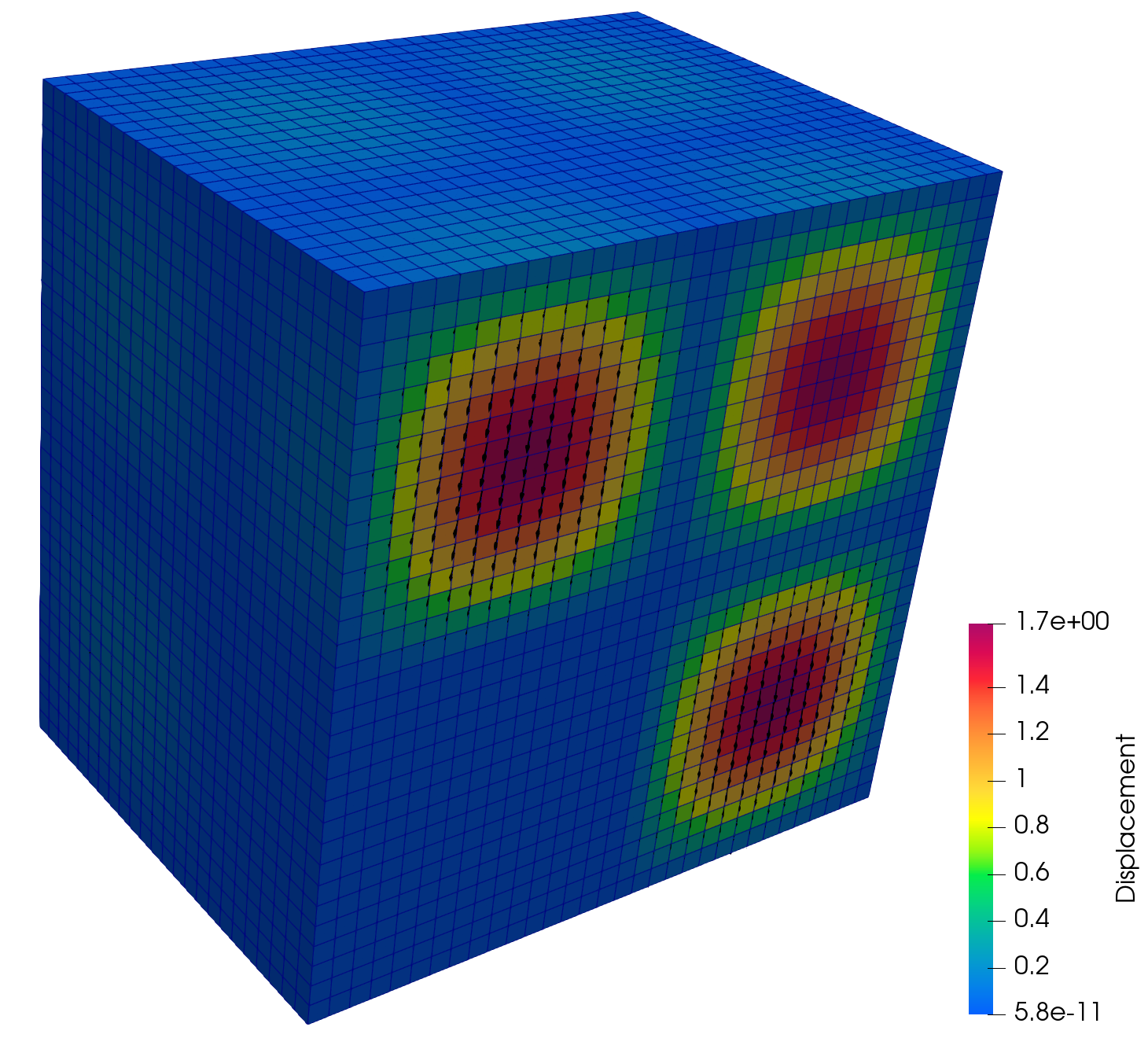}
        \caption{Displacement}
    \end{subfigure}
     \hspace{3mm}
        \begin{subfigure}{.30\textwidth}
        \centering
         \vspace{.3cm}
        \includegraphics[width=\linewidth]{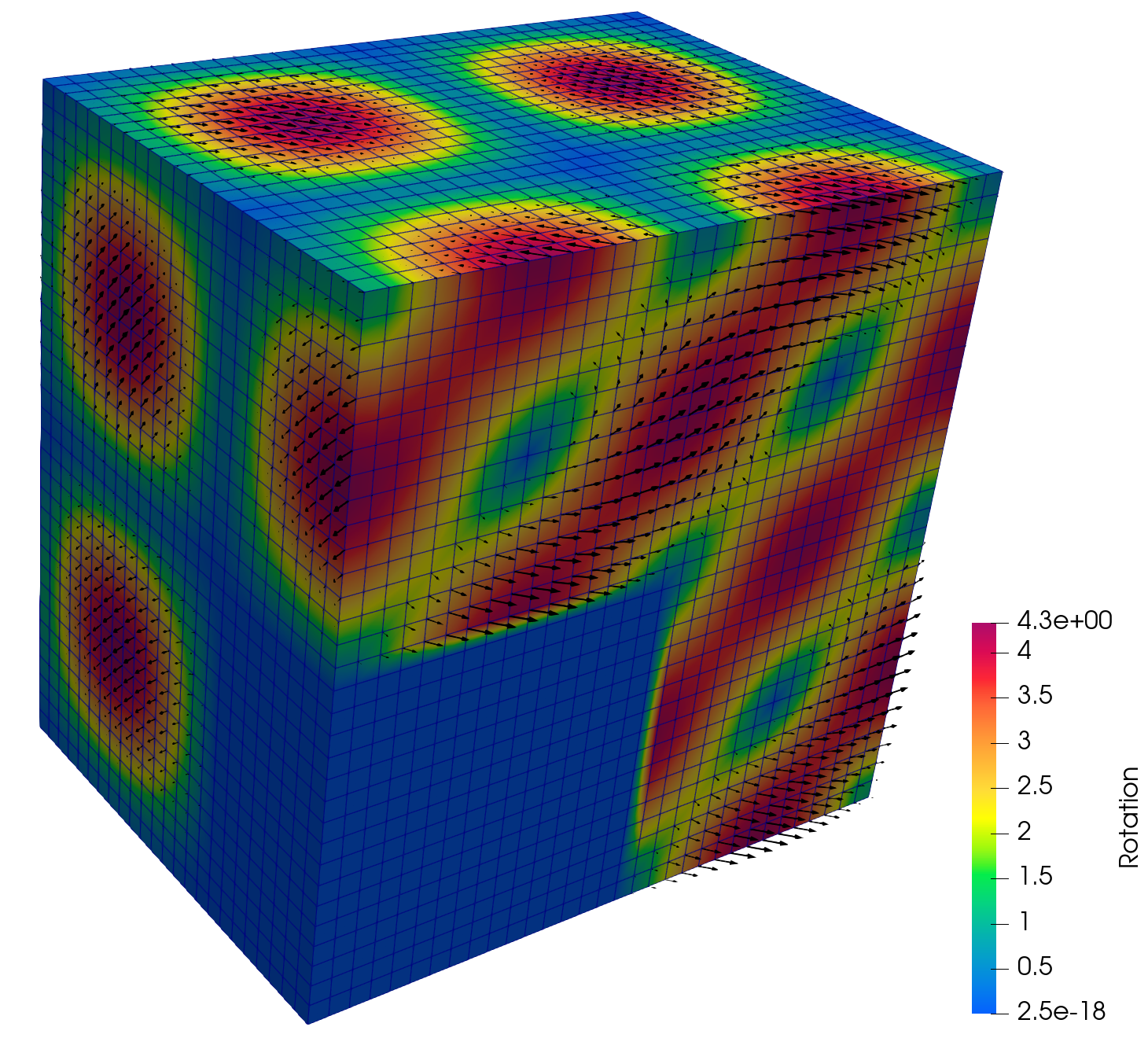}
        \caption{Rotation}
    \end{subfigure}
         \hspace{3mm}
        \begin{subfigure}{.30\textwidth}
        \centering
         \vspace{.3cm}
        \includegraphics[width=\linewidth]{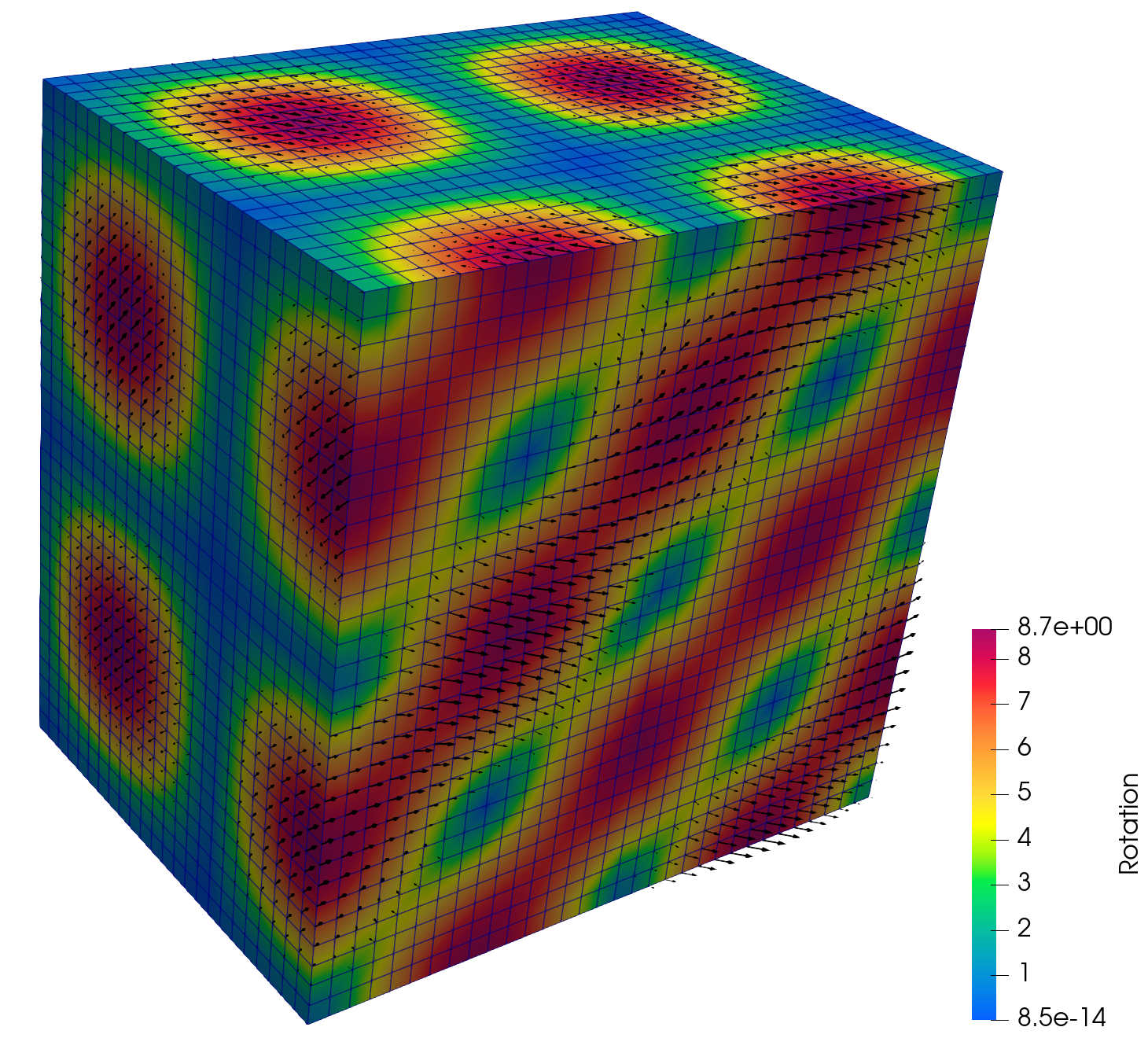}
        \caption{Scaled Rotation}
    \end{subfigure}
        \caption{Computed solution for Example 2 via the MSMFE-1 method, $h = 1/32$.}
         \label{ex2figure}
\end{figure}

\medskip
\noindent
{\bf Example 3: locking free property.} Our final example investigates the locking-free property of the MSMFE method. We focus on the MSMFE-1 method and use the analytical solution given in \eqref{ex1-soln}. The Lam\'e coefficients are determined from Young's modulus \(E\) and Poisson's ratio \(\nu\) via the well-known relationships:
\[
\lambda = \frac{E \nu}{(1 + \nu)(1 - 2 \nu)}, \quad \mu = \frac{E}{2(1 + \nu)}.
\]
We fix Young's modulus at \(E = 10^5\) and vary Poisson's ratio as \(\nu = 0.5 - k\), where \(k = 10^{-l}\) for \(l = \{1, 2, 5, 9\}\). Tables \ref{tab:convergence_table5}--\ref{tab:convergence_table8} display the errors and convergence rates for each case. If locking were present, the displacement solution would approach zero as \(\nu\) approaches 0.5. However, the results in Tables 5--8 indicate that this behavior does not occur, confirming the robustness of the method for nearly incompressible materials.

\begin{table}[H]
\centering
\caption{Relative errors and convergence rates for Example 3, $k=10^{-1}$.}
\vspace{-0.3cm}
\begin{adjustbox}{max width=\textwidth}
\begin{tabular}{|c|c|c|c|c|c|c|c|c|c|c|}
\specialrule{1pt}{0pt}{0pt}
\rowcolor{blue!20}
 & \multicolumn{2}{c|}{$\|\sigma-\sigma_h\|$} &  \multicolumn{2}{c|}{$\|\text{div}(\sigma-\sigma_h)\|$}&   \multicolumn{2}{c|}{$\|u-u_h\|$}&  \multicolumn{2}{c|}{$\|Q^u_hu-u_h\|$}& \multicolumn{2}{c|}{$\|\gamma-\gamma_h\|$} \\ 
\hline
\rowcolor{blue!20}
$h$ & Error & Rate & Error & Rate & Error & Rate & Error & Rate & Error & Rate \\
\hline
1/2 & $4.974\text{E}-01$ & $-$  & $2.826\text{E}-01$ & $-$ & $6.045\text{E}-01$ & $-$ & $6.668\text{E}-02$ & $-$ & $3.017\text{E}-01$ & $-$ \\
\hline
1/4 & $2.634\text{E}-01$ & $0.92$  & $1.464\text{E}-01$ & $0.95$ & $3.164\text{E}-01$ & $0.93$ & $1.429\text{E}-02$ & $2.22$ & $1.139\text{E}-01$ & $1.40$ \\
\hline
1/8 & $1.349\text{E}-01$ & $0.97$  & $7.405\text{E}-02$ & $0.98$ & $1.600\text{E}-01$ & $0.98$ & $3.496\text{E}-03$ & $2.03$ & $4.186\text{E}-02$ & $1.44$ \\
\hline
1/16 & $6.812\text{E}-02$ & $0.99$  & $3.715\text{E}-02$ & $1.00$ & $8.024\text{E}-02$ & $1.00$ & $8.803\text{E}-04$ & $1.99$ & $1.511\text{E}-02$ & $1.47$ \\
\specialrule{1pt}{0pt}{0pt}
\end{tabular}
\end{adjustbox}
\label{tab:convergence_table5}
\end{table}

\begin{table}[H]
\centering
\caption{Relative errors and convergence rates for Example 3, $k=10^{-2}$.}
\vspace{-0.3cm}
\begin{adjustbox}{max width=\textwidth}
\begin{tabular}{|c|c|c|c|c|c|c|c|c|c|c|}
\specialrule{1pt}{0pt}{0pt}
\rowcolor{blue!20}
 & \multicolumn{2}{c|}{$\|\sigma-\sigma_h\|$} &  \multicolumn{2}{c|}{$\|\text{div}(\sigma-\sigma_h)\|$}&   \multicolumn{2}{c|}{$\|u-u_h\|$}&  \multicolumn{2}{c|}{$\|Q^u_hu-u_h\|$}& \multicolumn{2}{c|}{$\|\gamma-\gamma_h\|$} \\ 
\hline
\rowcolor{blue!20}
$h$ & Error & Rate & Error & Rate & Error & Rate & Error & Rate & Error & Rate \\
\hline
1/2 & $6.618\text{E}-02$ & $-$  & $1.758\text{E}-01$ & $-$ & $6.074\text{E}-01$ & $-$ & $1.011\text{E}-01$ & $-$ & $3.017\text{E}-01$ & $-$ \\
\hline
1/4 & $3.039\text{E}-02$ & $1.12$  & $8.874\text{E}-02$ & $0.99$ & $3.167\text{E}-01$ & $0.94$ & $2.171\text{E}-02$ & $2.22$ & $1.139\text{E}-01$ & $1.40$ \\
\hline
1/8 & $1.497\text{E}-02$ & $1.02$  & $4.450\text{E}-02$ & $1.00$ & $1.601\text{E}-01$ & $0.98$ & $5.317\text{E}-03$ & $2.03$ & $4.187\text{E}-02$ & $1.44$ \\
\hline
1/16 & $7.418\text{E}-03$ & $1.00$  & $2.227\text{E}-02$ & $1.00$ & $8.025\text{E}-02$ & $1.00$ & $1.335\text{E}-03$ & $1.99$ & $1.511\text{E}-02$ & $1.47$ \\
\specialrule{1pt}{0pt}{0pt}
\end{tabular}
\end{adjustbox}
\label{tab:convergence_table6}
\end{table}

\begin{table}[H]
\centering
\caption{Relative errors and convergence rates for Example 3, $k=10^{-5}$.}
\vspace{-0.3cm}
\begin{adjustbox}{max width=\textwidth}
\begin{tabular}{|c|c|c|c|c|c|c|c|c|c|c|}
\specialrule{1pt}{0pt}{0pt}
\rowcolor{blue!20}
 & \multicolumn{2}{c|}{$\|\sigma-\sigma_h\|$} &  \multicolumn{2}{c|}{$\|\text{div}(\sigma-\sigma_h)\|$}&   \multicolumn{2}{c|}{$\|u-u_h\|$}&  \multicolumn{2}{c|}{$\|Q^u_hu-u_h\|$}& \multicolumn{2}{c|}{$\|\gamma-\gamma_h\|$} \\ 
\hline
\rowcolor{blue!20}
$h$ & Error & Rate & Error & Rate & Error & Rate & Error & Rate & Error & Rate \\
\hline
1/2 & $4.120\text{E}-02$ & $-$  & $1.585\text{E}-01$ & $-$ & $6.080\text{E}-01$ & $-$ & $1.072\text{E}-01$ & $-$ & $3.017\text{E}-01$ & $-$ \\
\hline
1/4 & $1.063\text{E}-02$ & $1.95$  & $7.966\text{E}-02$ & $0.99$ & $3.168\text{E}-01$ & $0.94$ & $2.312\text{E}-02$ & $2.21$ & $1.139\text{E}-01$ & $1.40$ \\
\hline
1/8 & $2.697\text{E}-03$ & $1.99$  & $3.988\text{E}-02$ & $1.00$ & $1.601\text{E}-01$ & $0.98$ & $5.695\text{E}-03$ & $2.02$ & $4.187\text{E}-02$ & $1.44$ \\
\hline
1/16 & $6.712\text{E}-04$ & $2.00$  & $1.995\text{E}-02$ & $1.00$ & $8.025\text{E}-02$ & $1.00$ & $1.434\text{E}-03$ & $1.99$ & $1.511\text{E}-02$ & $1.47$ \\
\specialrule{1pt}{0pt}{0pt}
\end{tabular}
\end{adjustbox}
\label{tab:convergence_table7}
\end{table}

\begin{table}[H]
\centering
\caption{Relative errors and convergence rates for Example 3, $k=10^{-9}$.}
\vspace{-0.3cm}
\begin{adjustbox}{max width=\textwidth}
\begin{tabular}{|c|c|c|c|c|c|c|c|c|c|c|}
\specialrule{1pt}{0pt}{0pt}
\rowcolor{blue!20}
 & \multicolumn{2}{c|}{$\|\sigma-\sigma_h\|$} &  \multicolumn{2}{c|}{$\|\text{div}(\sigma-\sigma_h)\|$}&   \multicolumn{2}{c|}{$\|u-u_h\|$}&  \multicolumn{2}{c|}{$\|Q^u_hu-u_h\|$}& \multicolumn{2}{c|}{$\|\gamma-\gamma_h\|$} \\ 
\hline
\rowcolor{blue!20}
$h$ & Error & Rate & Error & Rate & Error & Rate & Error & Rate & Error & Rate \\
\hline
1/2 & $4.120\text{E}-02$ & $-$  & $1.585\text{E}-01$ & $-$ & $6.080\text{E}-01$ & $-$ & $1.072\text{E}-01$ & $-$ & $3.017\text{E}-01$ & $-$ \\
\hline
1/4 & $1.064\text{E}-02$ & $1.95$  & $7.965\text{E}-02$ & $0.99$ & $3.168\text{E}-01$ & $0.94$ & $2.312\text{E}-02$ & $2.21$ & $1.139\text{E}-01$ & $1.40$ \\
\hline
1/8 & $2.680\text{E}-03$ & $1.99$  & $3.988\text{E}-02$ & $1.00$ & $1.601\text{E}-01$ & $0.98$ & $5.695\text{E}-03$ & $2.02$ & $4.187\text{E}-02$ & $1.44$ \\
\hline
1/16 & $6.713\text{E}-04$ & $2.00$  & $1.994\text{E}-02$ & $1.00$ & $8.025\text{E}-02$ & $1.00$ & $1.435\text{E}-03$ & $1.99$ & $1.511\text{E}-02$ & $1.47$ \\
\specialrule{1pt}{0pt}{0pt}
\end{tabular}
\end{adjustbox}
\label{tab:convergence_table8}
\end{table}

\section{Conclusion}
We presented two MFE methods for linear elasticity with weak stress symmetry on cuboid grids, which reduce to symmetric and positive definite cell-centered algebraic systems. These methods utilize the lowest order enhanced Raviart-Thomas space $\mathcal{ERT}_0$ for the weakly symmetric stress. The MSMFE-0 method reduces to a cell-centered scheme for displacements and rotations, while the MSMFE-1 method reduces to a cell-centered scheme for displacements only.  To prove the stability of the MSMFE-1 method, we developed a new $H(\curl)$-conforming matrix-valued space, which forms an exact sequence with the stress space, and established a discrete $\curl$-based inf-sup condition with the rotation space. We demonstrated that the resulting algebraic system for each method is symmetric and positive definite. We established first-order convergence for all variables in their natural norms and second-order convergence for the displacements at the cell centers. The theory was illustrated by numerical experiments. Additionally, we tested the performance of the methods for problems with discontinuous coefficients and parameters in the nearly incompressible regime, showing good performance. Possible extensions of this work include higher order methods \cite{high-order-mfmfe} and hexahedral elements \cite{IngWheYot-sym,WheXueYot-nonsym}.
  
\bibliographystyle{abbrv}
\bibliography{Project-1}

\end{document}